\documentclass[9pt, leqno]{jdg-p-1-2}

\numberwithin{equation}{section}
\newtheorem{theorem}{Theorem}
\newtheorem{lemma}[theorem]{Lemma}

\theoremstyle{definition}
\newtheorem{definition}[theorem]{Definition}

\newtheorem*{theorem*}{Theorem}
\newtheorem*{acknowledgement*}{Acknowledgement}

\newtheorem{proposition}[theorem]{Proposition}
\newtheorem{remark}[theorem]{Remark}

\newtheorem*{question*}{Question}
\newtheorem*{conjecture*}{Conjecture}

\newcommand{\gt}[0]{\tilde{g}}
\newcommand{\gb}[0]{\bar{g}}
\newcommand{\delt}[0]{\widetilde{\nabla}}
\newcommand{\del}[0]{\nabla}

\newcommand{\gam}[0]{\Gamma}
\newcommand{\gamt}[0]{\widetilde{\Gamma}}	

\newcommand{\Rm}[0]{\operatorname{Rm}}
\newcommand{\Rmt}[0]{\widetilde{\Rm}}
\newcommand{\Rmb}[0]{\overline{\Rm}}

\newcommand{\Rc}[0]{\operatorname{Rc}}
\newcommand{\Rct}[0]{\widetilde{\Rc}}

\newcommand{\defn}[0]{\doteqdot}
\newcommand{\ve}[1]{{\mathbf{#1}}}
\newcommand{\pd}[2]{\frac{\partial #1}{\partial #2}}
\newcommand{\pdt}[0]{\frac{\partial}{\partial t}}

\newcommand{\supp}[0]{\operatorname{supp}}
\newcommand{\trace}[0]{\operatorname{tr}}
\newcommand{\df}[0]{\doteqdot}
\newcommand{\dint}[1]{\iint\limits_{#1}}
\newcommand{\llangle}[0]{\left\langle}
\newcommand{\rrangle}[0]{\right\rangle}
\newcommand{\vol}[0]{\operatorname{vol}}

\begin{document}

\title{Backwards uniqueness of the Ricci flow}

%    Information for first author					

%    Address of record for the research reported here

\author{Brett Kotschwar}
\address{Department of Mathematics, MIT, Cambridge, MA 02139}
\thanks{This research was partially supported by NSF Grant DMS-0805834.}
\email{kotschwar@math.mit.edu}
%    Current address

%    \thanks will become a 1st page footnote.
%    Information for second author

%    General info

\date{April 2009}


\begin{abstract}
	In this paper, we prove a unique continuation or ``backwards-uniqueness''
theorem for solutions to the Ricci flow.  A particular consequence is that
the isometry group of a solution cannot expand within the lifetime of the solution.
\end{abstract}

\maketitle

\section{Introduction}

Let $M$ be a smooth $n$-dimensional manifold.  The Ricci flow, 
introduced by Richard Hamilton in \cite{H1}, is the evolution equation 
\begin{equation}\label{eq:rf}
	\pdt g = -2\Rc(g)
\end{equation}
for a family $g(t)$ of Riemannian metrics $g$ on $M$.  
Equation \eqref{eq:rf} is a quasilinear but only weakly-parabolic system --
the degeneracy being an artifact
of the diffeomorphism invariance of the Ricci tensor.
Properties such as short-time existence and uniqueness of solutions 
are therefore not direct consequences of standard parabolic theory.  
On closed $M$ these were established in \cite{H1} by
the use of a modified Nash-Moser inverse function theorem.  The proof of existence was
simplified shortly thereafter by DeTurck \cite{D} who observed that
via a judicious change of gauge, the question could be reduced to that for
an equivalent strictly parabolic system. His method was further parlayed 
into a similarly elegant
proof of uniqueness \cite{H3} which became the basis for later extensions of these results
to non-compact manifolds.  In this setting, Shi \cite{S} proved 
proved short-time existence to \eqref{eq:rf} for complete initial data of bounded curvature; 
uniqueness, in the category of complete solutions with uniformly bounded curvature, was proven by Chen-Zhu 
in the relatively recent \cite{CZ}.  
  
Our aim in this paper is to prove the following complementary uniqueness result.
\begin{theorem}\label{thm:rfbu}
  Let $M$ be an $n$-dimensional manifold.
  Suppose $g(t)$, $\gt(t)$ are complete solutions to \eqref{eq:rf} with 
  $|\Rm|_{g}\leq K$, $|\Rmt|_{\gt}\leq \tilde{K}$ on 
  $M\times [0, T]$.
  If $g(x, T)  = \gt(x, T)$, then $g(x, t) = \gt(x, t)$ on 
  $M\times [0, T]$.  
\end{theorem}

One consequence is that the flow does not sponsor the creation of new isometries
within the lifetime of a solution.  This answers a question of A. Fischer (cf. \cite{CLN}),
and, together with the uniqueness results of Hamilton \cite{H1} and 
Chen-Zhu \cite{CZ}, implies that the equation \eqref{eq:rf} 
preserves the isometry group of the solution.

\begin{theorem}\label{thm:isom}
	If $g$ is a solution to the Ricci flow with uniformly bounded curvature on $[0, T]$, then 
	$\operatorname{Isom}(g(t)) \subseteq \operatorname{Isom}(g(0))$ for all $0 \leq t \leq T$.
\end{theorem}

Another consequence is that a general solution cannot become a Ricci soliton in finite time.
\begin{theorem}\label{thm:soliton}
Suppose $g(t)$ is a complete solution to the Ricci flow on $[0, T]$ with
uniformly bounded curvature.  If $\bar{g} \df g(T)$ satisfies
\begin{equation}\label{eq:soliton}
	\Rc(\bar{g}) + \mathcal{L}_X\bar{g} + \frac{\lambda}{2} \bar{g} = 0
\end{equation}
for some $X\in C^{\infty}(TM)$ and $\lambda \in \mathbb{R}$,
then there exists $\varphi_t\in \operatorname{Diff}(M)$ and $c:[0, T]\to (0, \infty)$ with
$\varphi_T = Id$ and $c(T) = 1$ such that 
\begin{equation}\label{eq:selfsim}
	g(t) = c(t)\varphi_t^{*}(\bar{g})
\end{equation}
on $M\times [0, T]$. 
\end{theorem}

Equation \eqref{eq:selfsim} describes solutions which move solely under the actions of $\mathbb{R}_{+}$ and 
$\operatorname{Diff}(M)$, both symmetries of the equation \eqref{eq:rf}.  These \emph{self-similar} solutions
or \emph{Ricci solitons} are ubiquitous in the analysis of the long-time behavior of solutions,
arising, in many important cases, as renormalized limits about developing singularities and modeling 
the local geometry of highly-curved regions of the flow. They also frequently occur as critical cases
of inequalities and monotone formulas which hold for general solutions, most notably, perhaps, in Hamilton's differential
Harnack inequality \cite{H3} and Perelman's entropy formula \cite{P}.  As equation \eqref{eq:soliton} generalizes
the Einstein condition, Theorem \ref{thm:soliton} implies, in particular, that a solution cannot become Einstein 
in finite time.

Backwards-uniqueness and unique-continuation properties of solutions to parabolic equations and inequalities 
have been the objects of consistent study for at least half a century, beginning with the classical 
work of Mizohata \cite{M}, Yamabe \cite{Y}, Lees-Protter \cite{LP}, Agmon-Nirenberg\cite{AN}, and Landis-Oleinik
\cite{LO}, and extending through the work of Saut-Scheurer \cite{SS}, Lin \cite{L}, Sogge \cite{So}, Poon \cite{Po}, 
Escauriaza-Sverak-Veregin \cite{ESV}, and Escauriaza-Kenig-Ponce-Vega \cite{EKPV} over the past two decades. 
Whereas modern efforts have considerably weakened
the assumptions on the growth and regularity of the solution and the coefficients of the equation, 
due in in part to existing estimates on the derivatives of the
curvature tensor for solutions to \eqref{eq:rf}, we will need little of this recent innovation for our purposes.
In fact, the estimates we will ultimately prove are modelled on the classical
argument in Lees-Protter \cite{LP}.  The crucial point of this paper
is to reduce the question of backwards-uniqueness for the weakly parabolic system \eqref{eq:rf} 
to one for a pair of coupled differential inequalities amenable to these more-or-less standard techniques.  
The chief obstacle to overcome is that the usual means of converting
\eqref{eq:rf} to an equivalent parabolic system, namely the method of DeTurck, leads to an ill-posed problem
in application to backwards evolutions.

To highlight the problem with this conventional approach, let us briefly recall Hamilton's prescription \cite{H3} 
to prove uniqueness via DeTurck's method. (We will assume for simplicity that $M$ is compact.)
One begins with two solutions $g(t)$ and $\gt(t)$ to \eqref{eq:rf} with $g(0)=\gt(0)$, fixes
a background metric $\gb$, and solves two \emph{harmonic map heat flows} for diffeomorphisms
$\phi_t$, $\tilde{\phi}_t: M\to M$:
\begin{equation}\label{eq:hhf}
	\pd{\phi}{t} = \Delta_{g(t), \gb}\phi,\quad 
	\pd{\tilde{\phi}}{t} = \Delta_{\gt(t), \gb} \tilde{\phi}, \quad \phi_0 = \tilde{\phi}_0 = Id.
\end{equation} 
Here $\Delta_{g, \gb}$ represents the map Laplacian with respect to the domain metric $g$ and the target metric $\gb$.
The metrics $h(t) \df (\phi_t^{-1})^{*}g(t)$ and $\tilde{h}(t)\df (\tilde{\phi}_t^{-1})^*\gt(t)$, can then be seen to solve
the \emph{Ricci-DeTurck flow} relative to the background metric $\gb$; for $h$, this is 
\[
	\pdt h = -2\Rc(h(t)) + \mathcal{L}_{W[h,\gb]}h(t), 
	\quad\mbox{where}\quad W^k = h^{ij}\left(\Gamma[h]^k_{ij} - \Gamma[\gb]_{ij}^k\right), 
\]
and similarly for $\tilde{h}$.  The equations for $h(t)$ and $\tilde{h}(t)$ are strictly parabolic,
and since $h(0) = \tilde{h}(0)$, one can conclude from existing theory that $h(t) = \tilde{h}(t)$.

On the other hand, given $h(t)$ and $\tilde{h(t)}$, one can also recover $\phi_t$ and $\tilde{\phi}$ by means of the 
\emph{ordinary} differential equations
\[
	\pd{\phi}{t}(x, t) = -W[h(t), \gb](\phi_t(x), t), \quad\mbox{and}\quad 	
	\pd{\tilde{\phi}}{t}(x, t) = -W[\tilde{h}(t), \gb](\tilde{\phi}_t(x), t),
\]
which depend only on $h$, $\tilde{h}$ and $\gb$.  Thus, from ODE theory, we can conclude
that $\phi_t = \tilde{\phi}_t$, and hence that $g(t) = \phi_t^{*}(h(t)) = \tilde{\phi}_t^*(\tilde{h}(t)) = \tilde{g}(t)$.

The difficulty in applying this approach to the question of backwards uniqueness lies
in the the matter of obtaining from two solutions $g(t)$, $\tilde{g}(t)$ of \eqref{eq:rf} which
agree at some non-initial time $t= T$, two corresponding solutions of 
the Ricci-DeTurck flow $h(t)$, $\tilde{h}(t)$ with $h(T) = \tilde{h}(T)$. According to the scheme above,
we would need to seek solutions to \eqref{eq:hhf} with $\phi(T) = Id = \tilde{\phi}(T)$, i.e., to
solve the corresponding (and ill-posed) \emph{terminal-value} problem for the harmonic map heat flow.
Thus, it appears there is no straight-forward reduction of the question of the 
backwards uniqueness of the Ricci flow to the parabolic question of that for the Ricci-DeTurck flow.

Instead, we adapt an alternative technique of Alexakis \cite{A}, to which we were introduced by the recent
paper of Wong-Yu \cite{WY}. Rather than attempt to eliminate the degeneracy of the equation 
by means of a gauge transformation or otherwise, 
we focus instead on the strictly parabolic equations satisfied by the curvature tensor of a solution and its derivatives.
Of course, the equation satisfied by the difference of the curvature tensors of two solutions
(and of their covariant derivatives) will contain terms involving the lower-order differences of the metrics and connections.  We bundle
these quantities together with the differences of curvature and their derivatives in a larger system of
coupled differential inequalities--  the higher order terms satisfying a parabolic, ``partial differential'' inequality,
and the lower order terms satisfying an ``ordinary differential'' inequality.
For these inequalities -- upper bounds, effectively, for the heat operator applied to the higher-order differences
	and the time-derivative of the lower-order differences -- 
	we prove matching lower bounds in the form of Carleman-type estimates.  Here, and for the subsequent deduction
of backwards-uniqueness from these inequalities, we adapt the argument of \cite{LP} to fit a vector-bundle setting on potentially non-compact
manifolds. 

In \cite{A} (see also \cite{AIK}), Alexakis uses a similar technique to prove a unique continuation theorem 
for the vacuum Einstein equations, effectively reducing the problem to one amenable 
to existing Carleman estimates for hyperbolic equations.  In \cite{WY}, the authors apply Alexakis's method to prove 
a unique continuation property for solutions to a coupled Einstein-scalar field system for Riemannian metrics.  There
the method reduces the problem to a PDE-ODE system amenable to existing Carleman estimates for elliptic equations. 
Our method is an application of Alexakis's strategy to a degenerate parabolic geometric evolution equation.    
We carry out the construction of this associated system of differential inequalities in the next section, 
and prove a backwards-uniqueness theorem for a rather general geometric setting in the section following.  
In the last section, we take up the proofs of Theorems \ref{thm:rfbu}, \ref{thm:isom}, and \ref{thm:soliton} 
and discuss their application. 

\section{Reduction to a PDE-ODE system}

Let $g= g(t)$ and $\gt= \gt(t)$ be complete solutions to the Ricci flow on $M \times [0, T]$, 
and denote by $\del$ and $\delt$ their Levi-Civita connections and by 
$\Rm$ and $\Rmt$, their associated Riemannian curvature tensors.
Introduce the tensor fields $h \defn g - \gt$,   
$A \defn \del - \delt$, $B \defn \nabla A$,
$T \defn \Rm - \Rmt$ and $U \defn \del \Rm - \delt \Rmt$.   Here $A$ is a 
$(2, 1)$-tensor, given in local coordinates by 
$A_{ij}^k = \gam_{ij}^k - \gamt_{ij}^k$.
The strategy, following \cite{A}, \cite{WY}, 
is to group the evolving quantities
$h$, $A$, $B$, $T$, $U$ into two separate components of a coupled PDE-ODE system.  
We let 
\[
\mathcal{X} \defn T^1_3(M) \oplus T^1_4(M) \quad\mbox{and}\quad 
\mathcal{Y} \defn T_2(M)\oplus T_2^1(M)\oplus T^1_3(M),
\]
and group the evolving quantities as $\ve{X}(t)\defn T(t)\oplus U(t)\in \mathcal{X}$, and 
$\ve{Y}(t)\defn h(t)\oplus A(t)\oplus B(t)\in \mathcal{Y}$. The key observation
is that $X$ and $Y$ satisfy the following closed system of differential equalities, relative
to the metrics and connections induced on  $\mathcal{X}$ and $\mathcal{Y}$ by $g(t)$.

\begin{proposition}\label{prop:diffineq}
Suppose $g(t)$ and $\gt(t)$ are complete solutions to \eqref{eq:rf} on $M\times [0, T]$
with $g(T) = \gt(T)$ and
\[
	|\Rm|_{g(t)}\leq K,\quad |\Rmt|_{\gt(t)} \leq \tilde{K}
\]
for some constants $K$, and $\tilde{K}$.
Define the associated sections $\ve{X}$ and $\ve{Y}$ from $g$ and $\gt$
as above.
Then, for any $0 <\delta < T$, there exists a constant $C = C(\delta, K, \tilde{K}, T) >0$
such that
\begin{align}
\label{eq:diffineq1}
  \left|\left(\pdt - \Delta_{g(t)}\right) \ve{X}\right|^2_{g(t)}  
  &\leq  C\left(|\ve{X}|_{g(t)}^2 +| \nabla \ve{X}|_{g(t)}^2 
  +|\ve{Y}|_{g(t)}^2\right)\\
\label{eq:diffineq2}
  \left|\pdt \ve{Y}\right|_{g(t)}^2\ &\leq C\left(|\ve{X}|_{g(t)}^2 
  + |\nabla\ve{X}|_{g(t)}^2 + |\ve{Y}|_{g(t)}^2\right)
\end{align}
on $M \times [\delta, T]$.
\end{proposition}

\begin{remark}
	In fact, for two solutions $g(t)$, $\gt(t)$ to the Ricci flow, 
	we may omit the term $|\nabla \ve{X}|^2_{g(t)}$ from the left-hand side
	of \eqref{eq:diffineq1}, though, as we will see in the next section,
	its presence presents no additional complication. 
	It is crucial, however, that $\nabla Y$ not appear.
\end{remark}

We postpone the proof of Proposition \ref{prop:diffineq} to the end of the section
and begin by recalling
the evolution equations for quantities attached to a solution of the Ricci flow. 

\begin{lemma}
If $g(t)$ is a solution to \eqref{eq:rf}, then in local coordinates,
\begin{equation}
	\label{eq:evgam}	
	\pdt \Gamma_{ij}^k = -g^{mk}\left\{\nabla_i R_{jm} 
	+ \nabla_j R_{im} - \nabla_{m} R_{ij}\right\}
\end{equation}
\begin{align}\label{eq:evriem}
  \begin{split}
	\pdt R^{l}_{ijk} &= \Delta R^{l}_{ijk} 
	+ g^{pq}\left( R^r_{ijp}R^l_{rqk} - 2R^{r}_{pik}R^l_{jqr} 
	+ 2 R^l_{pir} R^r_{jqk}\right)\\
		&\phantom{=\Delta R^{l}_{ijk} } 
	- g^{pq}\left(R_{ip} R^l_{qjk} + R_{jp} R^l_{iqk} + R_{kp} R_{ijq}^l\right)  
	+ g^{pl}R_{pq} R^q_{ijk}
\end{split}
\end{align}
\begin{align}
\label{eq:evdriem} 	
\begin{split}
	\pdt \nabla_a R_{ijk}^l &= \Delta \nabla_a R_{ijk}^l 
	+ g^{pq}\nabla_a\left( R^r_{ijp}R^l_{rqk} - 2R^{r}_{pik}R^l_{jqr} 
	+ 2 R^l_{pir} R^r_{jqk}\right)\\
	&\phantom{=}
	-g^{pq}\left( R_{ip} \nabla_a R^l_{qjk} 
	- R_{jp} \nabla_aR^l_{iqk} - R_{kp} \nabla_a R_{ijq}^l\right)	 
	+ g^{pl}R_{pq} \nabla_aR^q_{ijk}
\end{split}
\end{align}
	
\end{lemma}

\begin{proof}  
For \eqref{eq:evgam}, \eqref{eq:evriem}, see e.g., \cite{CK}.
Then \eqref{eq:evdriem} follows from the first two and 
a commutation of derivatives.
\end{proof}

The simple observation that the size of the difference of $g^{-1}$ and $\gt^{-1}$
can be controlled by the norm of $h$ will be very useful for us.
Precisely, in local coordinates, we have
\begin{equation}
\label{eq:inv}	\gt^{ij} - g^{ij} =  \gt^{ia}g^{jb}h_{ab}.
\end{equation}
It will also be useful to record the following immediate consequence of compatibility:
\begin{equation}
\label{eq:delh} \nabla_c h_{ab} = A_{ca}^p\gt_{pb} + A_{cb}^p\gt_{ap}.
\end{equation}

Actually, for the rather indelicate inequalities
of Proposition \ref{prop:diffineq}, we will seldom need the precise expressions for 
the evolution equations, and we will adopt the convention 
that $V \ast_g W$  (or simply $V \ast W$) 
represents a linear combination of contractions of the tensors
$V$ and $W$ by the metric $g$.  Thus, for example, equations \eqref{eq:evriem}
and \eqref{eq:evdriem} may be economically expressed as 
\[
	\pdt \Rm = \Delta \Rm + \Rm \ast \Rm, \quad\mbox{ and }\quad
	\pdt \nabla \Rm =  \Delta \nabla \Rm + \Rm \ast \nabla \Rm, 
\]
and equations \eqref{eq:inv} and \eqref{eq:delh}, likewise, by
\[
	\gt^{-1} - g^{-1} = \gt^{-1}\ast h, \quad \mbox{ and } \quad \nabla h = \gt^{-1}\ast A.
\]

According to this convention, we may write the evolutions 
of the components of $\ve{X}$ and $\ve{Y}$ 
as follows.
\begin{lemma}\label{lem:evdiff}
If $g$, $\gt$, $h$, $A$, $B$, $T$, $U$ are as above, then
\begin{align}
\label{eq:hev}
	\pdt h_{ij} &= -2T_{lij}^l, \\
	\begin{split}
\label{eq:aev} 
	\pdt A_{ij}^k &= -g^{mk}\left(U_{ijmp}^p + U_{jimp}^p 
	 - U_{mijp}^p\right)\\ 
   	 &\phantom{=\quad\quad\quad}
	 + g^{kb}\gt^{ma}h_{ab}\left(\delt_{i}\tilde{R}_{jm} 
	 + \delt_{j}\tilde{R}_{im} - \delt_m\tilde{R}_{ij}\right),
 	\end{split}\\ 
	\begin{split}
\label{eq:bev}
 	\pdt B &= \nabla U  + h \ast A \ast \gt^{-1}\ast\delt\widetilde{\operatorname{Rm}}
			    + A \ast \gt^{-1}\ast \delt\widetilde{\operatorname{Rm}}\\
			&\phantom{=}\quad\quad\quad
			+ h \ast \gt^{-1}\ast \delt \delt\widetilde{\operatorname{Rm}}
			+ A \ast U + h \ast A\ast \gt^{-1}\ast\delt\Rmt\\
			&\phantom{=}\quad\quad\quad + A \ast \delt\Rmt + A\ast A\ast\gt^{-1}\ast \delt\Rmt,   
	\end{split}
\end{align}
and
\begin{align}
\begin{split}
\label{eq:tev}
	\left(\pdt - \Delta\right) T &=   h \ast \gt^{-1}\ast\delt\delt\Rmt + A\ast\delt\Rmt  
					+  T\ast \Rmt  + T\ast T,  \\
	&\phantom{=}\quad {}+ A\ast A \ast \Rmt  + B\ast \Rmt + h\ast \gt^{-1} \ast \Rmt \ast \Rmt,
\end{split}\\
\begin{split}
\label{eq:uev}
	\left(\pdt - \Delta\right) U &= 
	 h\ast \gt^{-1} \ast \delt^{(3)}\Rmt + A\ast\delt\delt\Rmt  + A\ast A \ast \delt\Rmt\\
	&\phantom{=}\quad {}+  B\ast \delt \Rmt  +h \ast\gt^{-1}\ast \Rmt \ast \delt \Rmt +  U \ast \Rmt\\
	&\phantom{=}\quad {}+ T\ast \delt \Rmt + T\ast U.
\end{split}
\end{align}
\end{lemma} Here $\Delta = g^{ij}\nabla_i\nabla_j$ represents the Laplacian induced by $g$ and $\nabla$ on the bundles
$T_3^1(M)$ and $T_4^1(M)$. 
\begin{proof}
Equation \eqref{eq:hev} is immediate from the definition of $h$,
and equation \eqref{eq:aev} follows from the evolution of the Christoffel symbols \eqref{eq:evgam} 
in view of \eqref{eq:inv}. From \eqref{eq:aev}, one obtains
\eqref{eq:bev} via
\[
	\pdt \nabla A = \nabla\pd{A}{t} + \pd{\Gamma}{t}\ast A, \quad \nabla_k\gt^{ij} = \gt^{ia}\gt^{jb}\nabla_k h_{ab},
\]
equation \eqref{eq:delh}, and that,  for any tensor $W$,
\[
	\delt W = \nabla W + A \ast W.
\]
For \eqref{eq:tev} and \eqref{eq:uev}, we use the above and  note, similarly, that 
\[ 
	\delt\delt W = \nabla \nabla W + A\ast \delt W + B \ast W + A\ast A\ast \delt W, 
\]
so
\[
	\widetilde{\Delta}\Rmt = \Delta \Rmt +  h \ast\gt^{-1}\ast \delt\delt\Rmt + A \ast\Rmt + B \ast\Rmt 
			+ A\ast A\ast \Rmt,
\]
with a corresponding version for $\widetilde{\Delta}\delt \Rmt$.  Then \eqref{eq:tev} and \eqref{eq:uev}
follow from \eqref{eq:evriem} and \eqref{eq:evdriem}. 
\end{proof}

Now we turn to the proof of the main result of the section.
\begin{proof}[Proof of Proposition \ref{prop:diffineq}]
By the estimates of W.X. Shi \cite{S}, the uniform bounds on $\Rm$ and $\Rmt$
imply that, for all $0 < \delta < T$ and $m \geq 0$, there exist constants 
$C_m = C_m(\delta, K, T)$ and $\tilde{C}_m = \tilde{C}_m(\delta, \tilde{K}, T)$ such that
\begin{equation}
\label{eq:shi}
   |\del^{(m)}\Rm|_{g(t)} \leq C_m, \quad\mbox{ and }\quad |\delt^{(m)}\Rmt|_{\gt(t)} \leq \tilde{C}_m
\end{equation}
on $M \times [\delta, T]$.  Moreover, the uniform bounds on $\Rc$ and $\Rct$
imply that the families $\{g(t)\}_{t\in[0, T]}$ and $\{\gt(t)\}_{t\in[0, T]}$ are uniformly equivalent 
and the equality of $g(T)$ and $\gt(T)$ implies that they are mutually so, i.e., that
there is a $\gamma = \gamma(K, \tilde{K}, T)$ such that
\[
	\gamma^{-1}g(x, t) \leq \gt(x, t) \leq \gamma g(x, t)
\]
on $M \times [0, T]$.  Thus we can replace the norms in the second inequality of \eqref{eq:shi} by $|\cdot|_{g(t)}$, 
provided we also replace
$\tilde{C}_m$ by an appropriate adjustment: $C_m \doteqdot \tilde{C}_m C^{\prime}(K, \tilde{K}, T) > 0$.  
It follows that the factors $\gt^{-1}$ and $\delt^{(m)}\Rmt$, $m\geq 0$, that appear in the right-hand side of the
evolution equations in Lemma \ref{lem:evdiff} are bounded with respect to $g(t)$ on $[\delta, T]$.  We claim that
$h$, $A$, $B$, $T$ and $U$ are also bounded on the interval and thus any extra-linear factors of these
quantities that appear in the evolution equations may be absorbed into the bounded coefficients.

The uniform equivalence, of course, implies that $|h|_{g(t)}$ is bounded, and
$|T|_{g(t)} = |\Rm -  \Rmt|_{g(t)}$ and $|U|_{g(t)} = |\nabla \Rm - \delt \Rmt|_{g(t)}$ are bounded by the discussion above.  
It remains to consider $A$ and $B = \nabla A$.

For $A$, using \eqref{eq:evgam} and the uniform equivalence of the $g(t)$, we have, for any $x\in M$,
\begin{align*}
	\left|A^{k}_{ij}(x, t)\right|_{g(t)} &= 
	\left|\gam^k_{ij}(x, t) - \gamt^k_{ij}(x, t)\right|_{g(t)}\\
	&\leq C^{\prime\prime}\int_t^T\left(|\nabla\Rc(x, s)|_{g(t)} +|\gt^{-1}(x, s)\ast\delt\Rct(x, s)|_{g(t)}\right)\,ds\\
	&\leq C^{\prime\prime\prime}T,  
\end{align*}
since $A(T) = 0$. Proceeding similarly, (and iteratively), we may bound $B$, and the higher derivatives $\nabla^{(m)} A$. 

The inequalities
\eqref{eq:diffineq1} and \eqref{eq:diffineq2} then follow from Lemma \ref{lem:evdiff}  and Cauchy-Schwarz.
\end{proof}

\section{Carleman inequalities and a general backwards-uniqueness theorem}
In this section, we prove a general backwards-uniqueness
theorem for time-dependent sections of vector bundles satisfying differential inequalities of the form \eqref{eq:diffineq1}, \eqref{eq:diffineq2}.  
As before, $M$ will denote, a smooth manifold of dimension $n$, equipped with a smooth family 
$\{g(t)\}_{t\in \mathcal{I}}$ of complete Riemannian metrics with Levi-Civita connections $\nabla = \nabla_{g(t)}$.
We will write $\pdt g= b$, and $B = \trace_g(b)$, so that 
$\pdt d\mu_g = (B/2) d\mu_g$, and
\begin{equation}\label{eq:evgamma}
	\pdt \Gamma_{ij}^k = \frac{1}{2}g^{mk}\left(\nabla_ib_{jm} + \nabla_j b_{im} - \nabla_m b_{ij}\right).
\end{equation} 

For simplicity, we will formulate our results for 
time-dependent sections of the tensor bundles $T^k_l(M)$ equipped with 
the metric and connection induced from $g$ and $\nabla$, although there is, of course,
an analogous statement for sections of general vector bundles equipped with families 
of metrics and compatible connections. 

In our setting, there is no harm in making the mild
abuse of notation of using $g$   
and $b$ to represent also the the induced metrics and their time-derivatives.  
Thus, for $V$, $W\in T_2^1(M)$, for example,
we will write
\[
 g(V,W) = \langle V, W\rangle = g^{ia}g^{jb}g_{ck}V_{ij}^k W_{ab}^c,
\]
and
\[
 b(V, W) = (g^{ia}g^{jb}b_{ck}- b^{ia}g^{jb}g_{ck} - g^{ia}b^{jb}g_{ck})V_{ij}^k W_{ab}^c.
\]  

Also, throughout the section, we will use $\Lambda\in C^{\infty}(S^2(M)\times {\mathcal{I}})$,
to denote a symmetric, positive-definite family of $(2, 0)$ tensors and
define from it the operators
\[
	\Box \df \Lambda^{ij}\nabla_i\nabla_j \quad\mbox{ and }\quad L\df \pdt - \Box.
\]

Our main objective is to prove the following theorem.

\begin{theorem}\label{thm:bu}
Let $\mathcal{X}$ and $\mathcal{Y}$ be finite direct sums of the bundles $T^k_l(M)$,
and 
$X \in C^{\infty}(\mathcal{X}\times[0, T])$, $Y\in C^{\infty}(\mathcal{Y}\times[0, T])$\
be smooth families of sections.
Assume that there exist positive constants $P$, $Q$, $\alpha_1$, and $\alpha_2$,  
such that
\begin{equation}\label{eq:pq}
	|b|_{g(t)}^2 + |\nabla b|_{g(t)}^2 \leq P,\qquad 
	\left|\pdt\Lambda\right|_{g(t)}^2 + |\nabla \Lambda|_{g(t)}^2 \leq Q, 
\end{equation}
and
\begin{equation}\label{eq:ell}
	\alpha_1 g^{ij}(x, t) \leq \Lambda^{ij}(x, t) \leq \alpha_2 g^{ij}(x, t),
\end{equation}
on $M\times[0, T]$ and that the metrics $g(t)$ are complete and satisfy 
\begin{equation}\label{eq:burcgrwth}
	\Rc(g(t)) \geq -K g(t).
\end{equation}
for some $K \geq 0$.
Further assume that the sections $X$, $Y$, obey the growth condition 
\begin{equation}\label{eq:bugrwth}
	\left|X(x,t)\right|_{g(t)}^2 + \left|(\nabla X)(x,t)\right|_{g(t)}^2 + \left|Y(x, t)\right|_{g(t)}^2 
	\leq A e^{a d_{g(t)}(x_0, x)},
\end{equation}
for $a$, $A > 0$ and some fixed $x_0\in M$,
as well as the system of differential inequalities
\begin{align}
\label{eq:pdeineq}
\left|LX\right|_{g(t)}^2 &\leq C \left(|X|_{g(t)}^2 + |\nabla X|_{g(t)}^2 + |Y|_{g(t)}^2\right),\\
\label{eq:odeineq}
\left|\pd{Y}{t}\right|_{g(t)}^2 &\leq C\left(|X|_{g(t)}^2 + |\nabla X|^2_{g(t)} + |Y|_{g(t)}^2\right)
\end{align}
for some $C \geq 0$.
Then $X(\cdot, T)\equiv 0$, $Y(\cdot, T) \equiv 0$ implies $X \equiv 0$, $Y\equiv 0$ on
$M\times [0, T]$. 
\end{theorem}

\begin{remark}
	It should be noted that Theorem \ref{thm:bu} is stated for application to a rather general geometric setting, and
	the growth conditions \eqref{eq:bugrwth} are less than optimal
	in many particular cases. For example, for the standard heat equation on $\mathbb{R}^n$
	(with $L = \pdt - \Delta$, $X(x,t) = u(x,t)$ and $Y(x,t) = 0$) the optimal condition
	is $|u(x, t)| \leq A e^{a|x|^2}$ on $\mathbb{R}^{n}$.  In fact, in \cite{EKPV}, it is shown that
	if this condition holds on $\mathbb{R}^{n}\times[0, T]$  and $u(x,T) \leq C_ke^{-k|x|^2}$ for all $k\geq 1$,
	then $u \equiv 0$.  We expect that the result of Theorem \ref{thm:bu} should admit
	similar improvements, even in its general setting.	
	
	Nevertheless, since the principal objects to which we apply Theorem \ref{thm:bu}
are the differences of the curvature tensors and their derivatives, the above theorem in its present form
will allow us to achieve the result of Theorem \ref{thm:rfbu} under reasonably weak conditions.
As the expression 
\[
		R_{a\bar{b}} = - \frac{\partial^2}{\partial z^a\partial\bar{z}^b}\log \det(g_{i\bar{j}}),
\]
for the Ricci curvature of a K\"{a}hler metric shows, the assumption of bounded curvature tensor
is, in a sense, analogous to the optimal growth rate for the corresponding result for the heat equation
on $\mathbb{R}^n$. 
\end{remark}

Our proof of Theorem \ref{thm:bu} will rely on Carleman-type estimates on 
members of the following subbundles.

\begin{definition}
Let $k_i$, $l_i$, $i=1,\ldots, N$ be non-negative integers,
$E \df \bigoplus_{i}T^{k_i}_{l_i}(M)$, and $\tau > 0$.
We will say that a family of sections $V\in C^{\infty}(E \times [0, \tau])$
belongs to $\mathcal{V} \doteqdot \mathcal{V}_M^{0, \tau}(E)$
if $V(\cdot , 0) \equiv 0 \equiv V(\cdot, \tau)$.
\end{definition}

Following \cite{LP}, we will use as kernel in our estimates 
integral powers of  
\[	
	\lambda(t)\df \lambda_{\tau,\eta}(t) \df \tau + \eta - t
\]
for $\tau$, $\eta > 0$.  For $\Omega \subset M$, 
we will write $\Omega_{\tau_1, \tau_2}\df \Omega \times [\tau_1, \tau_2]$ and
\[
	(V, W)_{\Omega_{\tau_1, \tau_2}}\df \dint{\Omega_{\tau_1, \tau_2}}\langle V, W\rangle_{g(t)}\,d\mu_{g(t)}\,dt,\quad
\|V\|_{\Omega_{\tau_1,\tau_2 }}^2 \df (V, V)_{\Omega_{\tau_1, \tau_2}}.
\]

\subsection{Local estimates}

We begin with local versions of the estimates on a precompact open set $\Omega\subseteq M$, 
working throughout on a fixed bundle $E = \bigoplus_{j=1}^{j=N}T^{k_j}_{l_j}(M)$ and suppressing
the dependency of the constants on the ranks $(k_j, l_j)$ as well as the dimension $n$.

The first estimate, a lower bound corresponding to the ODE portion \eqref{eq:odeineq} of our system 
of inequalities, is entirely elementary.

\begin{lemma}\label{lem:odecarleman}
For any $P_0 > 0$, there exist positive $T_1 = T_1(P_0)$ and $\eta_1 = \eta_1(P_0)$
such that if $0 < \tau < T_1$, $0 < \eta < \eta_1$, $V\in \mathcal{V}_{\Omega}^{0, \tau}(E)$,
and $g(t)$ is a smooth family of metrics on $\Omega\times[0, \tau]$
with $b = \pdt g$ as above satisfying $|b|^2 \leq P_0$ on $\Omega_{0, \tau}$,
then
\begin{equation}
\label{eq:ode}
\left\|\lambda^{-m}\sqrt{G}\pd{V}{t}\right\|^2_{\Omega_{0, \tau}}
	\geq m^2\left\|\lambda^{-(m+1)}\sqrt{G}V\right\|^2_{\Omega_{0, \tau}}
\end{equation}
for any non-negative $G\in C_c^{\infty}(\Omega)$ and $m\in \mathbb{N}$
\end{lemma}

\begin{proof}
Suppose $\tau$, $\eta > 0$, and $V \in \mathcal{V}^{0, \tau}_{\Omega}(E)$.
Let $\lambda =\lambda_{\tau, \eta}$ as before and fix $m\in \mathbb{N}$.  
Then $Z \doteqdot \lambda^{-m} V\in {\mathcal{V}}^{0, \tau}_{\Omega}(E)$ 
also, and
\[
 \pd{Z}{t} = m\lambda^{-(m+1)} V + \lambda^{-m}\pd{V}{t},
\]
so
\begin{equation}\label{eq:zident}
  \left|\lambda^{-m}\pd{V}{t}\right|^2 = \left|\pd{Z}{t}\right|^2 
	- 2m \left\langle \lambda^{-1}Z, \pd{Z}{t}\right\rangle 
	+ m^2\left|\frac{Z}{\lambda}\right|^2,
\end{equation}
	and
\begin{align}
\begin{split}
\label{eq:odetemp1}
  -2m \left\langle \frac{Z}{\lambda}, 
  \pd{Z}{t}\right\rangle G\,d\mu 
   &= -\pdt \left(m\lambda^{-1}|Z|^2  G d\mu \right)\\
&\phantom{=}\quad
    +m\left( |\lambda^{-1} Z|^2 +  \lambda^{-1}b(Z, Z) + \lambda^{-1}
    \frac{B}{2}|Z|^2\right)G\,d\mu.
\end{split}
\end{align}
Since $\lambda^{-1}Z = \lambda^{-m-1}V$, and $\lambda(t) \leq \tau + \eta$,
we have
\begin{equation}\label{eq:odetemp2}
 |\lambda^{-1} Z|^2 +  \lambda^{-1}b(Z, Z) + \lambda^{-1}
    \frac{B}{2}|Z|^2 \geq (1 - C^{\prime}(\tau + \eta))\left|\lambda^{-m-1}V\right|^2.  
\end{equation}
for some $C^{\prime} = C^{\prime}(P_0) > 0$
Combining \eqref{eq:odetemp1} and \eqref{eq:odetemp2}
with \eqref{eq:zident}, we obtain
\[
\left\|\lambda^{-m}\sqrt{G}\pd{V}{t}\right\|^2_{\Omega_{0, \tau}}
	\geq m^2\left(1 + \frac{1 - C^{\prime}(\tau + \eta)}{m}\right) 
	\left\|\lambda^{-(m+1)}\sqrt{G}V\right\|^2_{\Omega_{0, \tau}}
\]
upon integration.  Thus choosing $T_1 + \eta_1 < 1/ C^{\prime}$, we may ensure
\eqref{eq:ode} for $\tau < T_1$, $\eta < \eta_1$.   
\end{proof}

For the proof of the next two lemmas, we follow the basic outline
of the argument in \cite{LP}, making adjustments for the 
vector-bundle setting, the time-dependency of the metric, connection,  and measure,
and the cut-off function $G$.

\begin{lemma}\label{lem:pde1}
For any $\alpha_2$, $P_0$, $Q_0 > 0$, there  exist positive constants $C_1$, $T_2$, and $\eta_2$
depending only on this data such that if $0 < \tau < T_2$, 
and $g(t)$, $\Lambda(t)$
satisfy
\[
 \Lambda^{ij} \leq \alpha_2 g^{ij},\quad	|b|_{g(t)}^2 + |\nabla b|_{g(t)}^2 \leq P_0, \quad 
	\left|\pd{\Lambda}{t}\right|_{g(t)}^2 + \left|\nabla \Lambda\right|_{g(t)}^2  \leq Q_0,
\]
on $\Omega_{0,\tau}$, then
\begin{align}
\label{eq:pde1}
\begin{split}
	&\|\lambda^{-m}\sqrt{G}L V\|_{\Omega_{0, \tau}}^2 
		+ C_1\|\lambda^{-m}\sqrt{G}\nabla V\|_{\Omega_{0, \tau}}^2\quad\\ 
	&\qquad\qquad\qquad\geq 
	\frac{m}{2}
		\|\lambda^{-(m+1)}\sqrt{G}V\|^2
	- C_1\!\!\!\dint{\supp{G}\times[0, \tau]}\!\!\!\lambda^{-2m}|\nabla V|^2\frac{|\nabla G|^2}{G}\,d\mu\,dt
\end{split}
\end{align}
for any $m\in \mathbb{N}$, $V \in \mathcal{V}_{\Omega}^{0, \tau}(E)$, $0 < \eta < \eta_2$, and non-negative 
$G\in C_c^{\infty}(\Omega)$.
\end{lemma}
\begin{proof}
Fix $\tau > 0$, and let $V\in \mathcal{V}_{\Omega}^{0, \tau}(E)$. Define $Z = \lambda^{-m}V$ as before. 
We have
\[
  \lambda^{-m} LV   = LZ   -m\lambda^{-1}Z = \pd{Z}{t} -\Box Z - m \lambda^{-1}Z.
\]
So,
\begin{align*}
\begin{split}
	|\lambda^{-m}LV|^2 &= \left|\pd{Z}{t}\right|^2 + \left|\Box Z + m \lambda^{-1}Z\right|^2
	- 2\llangle \Box Z, \pd{Z}{t}\rrangle - 2m\llangle \lambda^{-1}Z, \pd{Z}{t}\rrangle,  
\end{split}
\end{align*}
and 
\begin{align}\label{eq:pde1main}
\begin{split}
	&\|\lambda^{-m}\sqrt{G}LV\|_{\Omega_{0,\tau}}^2 \\
	&\qquad\qquad\geq 
	\left\|\sqrt{G}\pd{Z}{t}\right\|_{\Omega_{0, \tau}}^2 
	- 2\left(G\Box Z, \pd{Z}{t}\right)_{\Omega_{0, \tau}} - 2m \left(G\lambda^{-1}Z, \pd{Z}{t}\right)_{\Omega_{0, \tau}}.
\end{split}
\end{align}
We now proceed to estimate the integrands in the last two terms in \eqref{eq:pde1main} from below.
First, we have the identity
\begin{align*}
-2\left\langle \Box Z, \pd{Z}{t}\right\rangle &= -2\nabla_a\left(\Lambda^{ab}\left\langle \nabla_b Z, \pd{Z}{t}\right\rangle\right)
		+2\nabla_a\Lambda^{ab}\llangle \nabla_b Z, \pd{Z}{t}\rrangle	
	+2\Lambda^{ab}\left\langle \nabla_b Z, \nabla_a \pd{Z}{t}\right\rangle\\
\begin{split}
&= -2\nabla_a\left(\Lambda^{ab}\llangle\nabla_b Z, \pd{Z}{t}\rrangle\right)	+2\nabla_a\Lambda^{ab}\llangle \nabla_b Z, \pd{Z}{t}\rrangle
 +2\Lambda^{ab}\left\langle\nabla_b Z, \left[\nabla_a, \pdt\right]Z\right\rangle\\
  &\phantom{=}\qquad\qquad+\pdt\left(\Lambda^{ab}\langle \nabla_a Z, \nabla_b Z\rangle\right) 
- \pd{\Lambda^{ab}}{t}\langle\nabla_a Z, \nabla_b Z\rangle - \Lambda^{ab}b(\nabla_a Z, \nabla_b Z). 
\end{split}
\end{align*}
Thus,
\begin{align}
\label{eq:pde1pen1}
\begin{split}
&\!\!\!\!\!\!\!\!\!\!\!\!\!\!-2\left\langle \Box Z, \pd{Z}{t}\right\rangle G\,d\mu_g 
= -2 \nabla_a\left(\Lambda^{ab}\left\langle\nabla_b Z, \pd{Z}{t}\right\rangle G\right)\,d\mu_g\\
&\qquad{}+\pdt\left[\Lambda^{ab}\langle \nabla_a Z, \nabla_b Z\rangle G\,d\mu_g\right] 
	+ \left\{2\Lambda^{ab}\left\langle \nabla_b Z, \left[\nabla_a, \pdt\right] Z\right\rangle\right.\\
&\qquad\left.{} - \Lambda^{ab}b(\nabla_a Z, \nabla_b Z) 
	 - \pd{\Lambda^{ab}}{t}\langle \nabla_a Z, \nabla_b Z\rangle
	-\frac{B}{2}\Lambda^{ab}\langle \nabla_a Z, \nabla_b Z\rangle\right\} G\, d\mu_g\\
&\qquad\qquad\qquad{} 
	+ 2\left(\Lambda^{ab}\nabla_a G + G\nabla_a\Lambda^{ab}\right)\left\langle \nabla_b Z, \pd{Z}{t}\right\rangle \, d\mu_g.
\end{split}
\end{align}
Now, in view of \eqref{eq:evgamma},
\[
	\left|\left[\nabla, \pdt\right] Z\right| \leq C^{\prime}|\nabla b||Z| \leq C^{\prime}\sqrt{P_0} |Z|, 
\]
so, using Cauchy's inequality we may estimate the quantity in brackets in \eqref{eq:pde1pen1} by
\[
	\big\{\cdots\big\} \geq  - C^{\prime\prime}\left(|Z|^2 + |\nabla Z|^2\right)
\]
for $C^{\prime\prime} = C^{\prime\prime}(\alpha_2, P_0, Q_0) > 0$.
Similarly, we estimate the factor in the last term of \eqref{eq:pde1pen1} as
\begin{align*}
	2(G\nabla_a\Lambda^{ab} + \Lambda^{ab}\nabla_a G)\left\langle \nabla_b Z, \pd{Z}{t}\right\rangle 
	&\geq - \left|\pd{Z}{t}\right|^2 G 
		-2\left(|\nabla\Lambda|^2G  + |\Lambda|^2\frac{|\nabla G|^2}{G}\right)|\nabla Z|^2\\
&\geq - \left|\pd{Z}{t}\right|^2 G - 2\left(Q_0 G +  n\alpha_2^2\frac{|\nabla G|^2}{G}\right)|\nabla Z|^2	
\end{align*}
on $\supp{G}\times[0, \tau]$.
Taking this into account when integrating \eqref{eq:pde1pen1} over $\Omega_{0, \tau}$,
we obtain
\begin{align}
\begin{split}
\label{eq:pde1pen2}
	-2\left( G\Box Z, \pd{Z}{t} \right)_{\Omega_{0, \tau}} &\geq
		-\left\|\sqrt{G}\pd{Z}{t}\right\|^2_{\Omega_{0, \tau}} 
		-C^{\prime\prime\prime} \left(\|\sqrt{G}{Z}\|_{\Omega_{0, \tau}}^2
			+ \|\sqrt{G}\nabla Z\|_{\Omega_{0, \tau}}^2\right)\\ 
			&\qquad\phantom{\geq}{}-C^{\prime\prime\prime} \!\!\!
		\dint{\supp{G}\times[0, \tau]}\!\!\!\frac{|\nabla G|^2}{G}|\nabla Z|^2\,d\mu_g\,dt	
\end{split}	
\end{align}
for $C^{\prime\prime\prime} = C^{\prime\prime\prime}(\alpha_2, P_0, Q_0)$.

For the last term in \eqref{eq:pde1main}, we rearrange terms as in Lemma \ref{lem:odecarleman}
to obtain
\begin{align*}
\begin{split}
-2m\left\langle \lambda^{-1}Z, \pd{Z}{t}\right\rangle G\, d\mu_g 
	&= -m\pdt\left(\lambda^{-1}|Z|^2 G \, d\mu_g\right)\\
	&\quad + m\lambda^{-2} \left(\left(1  + \lambda \frac{B}{2}\right)|Z|^2 + \lambda b(Z, Z)\right) G\, d\mu_g. 
\end{split}
\end{align*}
Integrating over $\Omega_{0, \tau}$, and using $0\leq \lambda(t) \leq \tau + \eta$, we have
\begin{align}\label{eq:pde1last1}
\begin{split}
 -2m\left(\lambda^{-1}G Z, \pd{Z}{t}\right)_{\Omega_{0, \tau}} &\geq
	m\left(1 - C^{\prime\prime\prime\prime}(\tau + \eta)\right) \|\lambda^{-1}\sqrt{G}Z\|^2_{\Omega_{0, \tau}}, 	
\end{split}
\end{align}
where $C^{\prime\prime\prime\prime} = C^{\prime\prime\prime\prime}(P_0)$.

Inserting \eqref{eq:pde1pen2} and \eqref{eq:pde1last1} into \eqref{eq:pde1main}
and using $|Z|^2 \leq(\tau + \eta)^2|\lambda^{-(m+1)}V|^2$, we have
\begin{align*}
\begin{split}
	&\|\lambda^{-m}\sqrt{G}L V\|_{\Omega_{0, \tau}}^2 
		+ C^{\prime\prime\prime}\|\lambda^{-m}\sqrt{G}\nabla V\|_{\Omega_{0, \tau}}^2\quad\\ 
	&\qquad\qquad\qquad\geq 
	m\left(1 - C^{\prime\prime\prime\prime}(\tau + \eta) - C^{\prime\prime\prime}\frac{(\tau + \eta)^2}{m}\right)
		\|\lambda^{-(m+1)}\sqrt{G}V\|^2  \\
&\phantom{\geq}\qquad\qquad\qquad\qquad 
	- C^{\prime\prime\prime}\!\!\!\dint{\supp{G}\times[0, \tau]}\!\!\!\lambda^{-2m}|\nabla V|^2\frac{|\nabla G|^2}{G}\,d\mu\,dt.
\end{split}
\end{align*}
Finally, choosing $T_2$ and $\eta_2$ so that  
\[
	 C^{\prime\prime\prime\prime}(T_2 + \eta_2) + C^{\prime\prime\prime}(T_2 + \eta_2)^2 < \frac{1}{2},
\]
we have \eqref{eq:pde1}.
\end{proof}

With this result in hand, we turn to the key lemma.  
\begin{lemma}\label{lem:pde2}
	For any positive $\alpha_1$, $\alpha_2$, $P_0$, and $Q_0$,
	there exist constants $C_2$, $m_1$, $T_3$, $\eta_4$ depending only on
	this data, such that if $0 < \tau < T_3$,
	and $g(t)$, $\Lambda(t)$ satisfy
\[
	\alpha_1 g^{ij}(x,t) \leq \Lambda^{ij} \leq \alpha_2 g^{ij}(x,t),
\]
and
\[
	|b|^2_{g(t)} + |\nabla b|^2_{g(t)} \leq P_0, \qquad\quad 
	\left|\pd{\Lambda}{t}\right|^2 + |\nabla \Lambda|^2_{g(t)} \leq Q_0,
\] 
on $\Omega_{0, \tau}$, then
\begin{align}\label{eq:pde2}
\begin{split}
	&\rho_m(\tau, \eta)\|\lambda^{-m}\sqrt{G}LV\|_{\Omega_{0, \tau}}^2 
		\geq \|\lambda^{-(m+1)}\sqrt{G}V\|_{\Omega_{0, \tau}}^2
		+\frac{1}{2}\|\lambda^{-m}\sqrt{G}\nabla V\|_{\Omega_{0, \tau}}^2\\
	&\qquad\qquad \phantom{=}{}- \rho_m(\tau, \eta)\!\!\!\dint{\supp G\times[0, \tau]}\!\!\! 
		\frac{|\nabla G|^2}{G}\left(|\lambda^{-m}V|^2 + |\lambda^{-(m+1)}\nabla V|^2\right)\,d\mu\,dt
\end{split}
\end{align}
 for any $V\in \mathcal{V}^{0, \tau}_{\Omega}(E)$, $m \geq m_1$, $\eta < \eta_3$, and non-negative
$G\in C^{\infty}_{c}(\Omega)$,
where
\[
	\rho_m(\tau, \eta) \df C_2\left(\frac{1}{m} +  (\tau + \eta) + \left(\frac{m+1}{m}\right)(\tau + \eta)^2\right).
\]
\end{lemma}
\begin{proof}
Fix $\tau > 0$ and let $V\in \mathcal{V}_{\Omega}^{0, \tau}(E)$.  We begin with the identity
\begin{equation}\label{eq:pde2main}
\langle \lambda^{-(m+1)} V, \lambda^{-m+1}LV \rangle
	= \left\langle \lambda^{-2m}V, \pd{V}{t}\right\rangle - \langle \lambda^{-2m}V, \Box V\rangle.
\end{equation}
Multiplying the first term on the right by the volume form and recasting it as 
in the previous two lemmas, we obtain
\begin{align*}
\begin{split}
 \left\langle \lambda^{-2m} V, \pd{V}{t}\right\rangle G\,d\mu_g   &= 
	\pdt\left\{\frac{1}{2}|\lambda^{-m}V|^2 G d\mu_g\right\} -m\lambda|\lambda^{-(m+1)}V|^2 G\, d\mu_g\\
	&\phantom{=}\quad -\frac{\lambda^{-2m}}{2}\left(b(V,V) 
	+ \frac{B}{2}|V|^2\right)G\, d\mu_g.
\end{split}		
\end{align*}
Integrating and  using $\lambda \leq \tau + \eta$, we have
\begin{align}\label{eq:pde2first}
\begin{split}
	\left(\lambda^{-2m} G V, \pd{V}{t}\right)_{\Omega_{0, \tau}} 
	&\geq -\left(m(\tau + \eta) 
	+ C^{\prime}(\tau + \eta)^2\right)\|\lambda^{-(m+1)}\sqrt{G}V\|^2_{\Omega_{0, \tau}}, 
\end{split}	
\end{align}
for $C^{\prime} = C^{\prime}(P_0) > 0.$

Likewise, for the second term on the right side of \eqref{eq:pde2main}, we 
find
\begin{align}\label{eq:pde2last1}
\begin{split}
	&\!\!\!\!\!\!\!\!-\langle\lambda^{-2m}V, \Box V\rangle G \,d\mu_g\\
	&\qquad\qquad={} -\lambda^{-2m}\nabla_a\left(\Lambda^{ab}\langle\nabla_b V, V\rangle G\right)\,d\mu_g
		+ \lambda^{-2m}\Lambda^{ab}\langle \nabla_a V, \nabla_b V\rangle G\, d\mu_g\\
&\qquad\qquad\qquad\qquad {}+\lambda^{-2m}\Lambda^{ab} \nabla_a G\langle V, \nabla_b V\rangle\,d\mu_g
	+ \lambda^{-2m}\nabla_a\Lambda^{ab}\langle V, \nabla_b V\rangle\,d\mu_g
\end{split}
\end{align}
On account of the uniform ellipticity of $\Lambda$, the second term on the right-hand side of \eqref{eq:pde2last1} satisfies
\begin{equation}\label{eq:pde2last4}
	\dint{\Omega_{0,\tau}}\lambda^{-2m}\Lambda^{ab}\langle\nabla_a V, \nabla_b V\rangle G\, d\mu_g\, dt
	\geq \alpha_1\|\lambda^{-m}\sqrt{G}\nabla V\|^2_{\Omega_{0, \tau}}.
\end{equation}
upon integration.
Integrating the third term on the right side of \eqref{eq:pde2last1} and using Cauchy's inequality, we have
\begin{align}\label{eq:pde2last2}
\begin{split}
	&\dint{\Omega_{0, \tau}}\Lambda^{ab}\nabla_a G\langle V, \nabla_b V\rangle\,d\mu_g\,dt \\
	&\qquad\qquad\geq -\frac{\alpha_1}{4}\|\lambda^{-m}\sqrt{G}\nabla V\|_{\Omega_{0, \tau}}^2 
	- n\frac{\alpha_2^2}{\alpha_1}\!\!\!\dint{\supp{G}\times[0, \tau]}\!\!\!\lambda^{-2m}\frac{|\nabla G|^2}{G}|V|^2\,d\mu_g\,dt.
\end{split}
\end{align}
Similarly, for the last term in \eqref{eq:pde2last1},
\begin{align}\label{eq:pde2last3}
\begin{split}
	&(\lambda^{-2m}G V, \nabla_a\Lambda^{ab}\nabla_b V)_{\Omega_{0, \tau}}\\
&\qquad\qquad\geq - \frac{\alpha_1}{4}\|\lambda^{-m}\sqrt{G}\nabla V\|^2_{\Omega_{0, \tau}} 
     -(\tau + \eta)^2\frac{Q_0}{\alpha_1}\|\lambda^{-(m+1)}\sqrt{G}V\|^2_{\Omega_{0, \tau}}.
\end{split}
\end{align}

Then, integrating \eqref{eq:pde2main} over $\Omega_{0, \tau}$ and applying equations \eqref{eq:pde2first},
 \eqref{eq:pde2last1}, and \eqref{eq:pde2last4}, we obtain
\begin{align*}
\begin{split}
 	&(\lambda^{-(m+1)} \sqrt{G}V, \lambda^{-m+1}\sqrt{G}L V)_{\Omega_{0, \tau}}\\
&\qquad\qquad\geq
		\frac{\alpha_1}{2}\|\lambda^{-m}\sqrt{G}\nabla V\|_{\Omega_{0, \tau}}^2
	- \left(m(\tau+\eta) + C^{\prime\prime}(\tau +\eta)^2\right)\|\lambda^{-(m+1)}\sqrt{G}V\|_{\Omega_{0, \tau}}^2\\
&\qquad\qquad\qquad\qquad
		{}- n\frac{\alpha_2^2}{\alpha_1}\!\!\!\dint{\supp{G}\times[0, \tau]}\!\!\!\lambda^{-2m}\frac{|\nabla G|^2}{G}|V|^2\,d\mu_g\,dt,
\end{split}		
\end{align*}
for $C^{\prime\prime} = C^{\prime\prime}(\alpha_1, \alpha_2, P_0, Q_0)>0$.

Estimating the left-hand side of this inequality from above, we have
\begin{align*}
	(\lambda^{-(m+1)} \sqrt{G}V, \lambda^{-m+1}\sqrt{G}L V)_{\Omega_{0, \tau}}
	&\leq \frac{1}{2}\|\lambda^{-(m+1)}\sqrt{G} V\|_{\Omega_{0, \tau}}^2
	 + \frac{1}{2}\|\lambda^{-m+1}\sqrt{G} L V\|_{\Omega_{0, \tau}}^2\\
	&\leq  \frac{1}{2}\|\lambda^{-(m+1)}\sqrt{G} V\|_{\Omega_{0, \tau}}^2
	 + \frac{(\tau + \eta)^2}{2}\|\lambda^{-m}\sqrt{G} L V\|_{\Omega_{0, \tau}}^2,
\end{align*}
so, put together, 
\begin{align}\label{eq:pde2lower2}
\begin{split}
	&\frac{\alpha_1}{2}\|\lambda^{-m}\sqrt{G} \nabla V\|_{\Omega_{0,\tau}}^2 \leq h_m(\tau, \eta)
		\|\lambda^{-(m+1)}\sqrt{G}V\|_{\Omega_{0, \tau}}^2 \\
&\qquad {} + \frac{(\tau +\eta)^2}{2}
	\|\lambda^{-m}\sqrt{G}LV\|_{\Omega_{0, \tau}}^2 +  
	n\frac{\alpha_2^2}{\alpha_1}\!\!\!\dint{\supp{G}\times[0,\tau]}\!\!\!\lambda^{-2m}\frac{|\nabla G|^2}{G}|V|^2\,d\mu_g\,dt
\end{split}
\end{align}
where
\[
	h_m(\tau, \eta) \df \frac{1}{2} + m(\tau + \eta) + C^{\prime\prime}(\tau + \eta)^2.
\]

Now we apply Lemma \ref{lem:pde1}, which supplies $C_1$,
$T_2$, $\eta_2$, and  depending on $\alpha_2$, $P_0$, and $Q_0$ such that
\begin{align*}
\begin{split}
& \frac{2}{m} \|\lambda^{-m}\sqrt{G}LV\|_{\Omega_{0, \tau}}^2
	+ \frac{2C_1}{m}\|\lambda^{-m}\sqrt{G}\nabla V\|_{\Omega_{0, \tau}}^2
+ \frac{2C_1}{m}\!\!\!\dint{\supp{G}\times[0, \tau]}\!\!\!\lambda^{-2m}\frac{|\nabla G|^2}{G}|\nabla V|^2\,d\mu\,dt\\
&\qquad\qquad\qquad\qquad\qquad\qquad\qquad\geq
	\|\lambda^{-(m+1)}\sqrt{G} V\|_{\Omega_{0, \tau}}^2
\end{split}
\end{align*}  
for all $m \in \mathbb{N}$, if $\tau < T_2$ and $\eta < \eta_2$.
Inserting this into \eqref{eq:pde2lower2}, we have
\begin{align*}
\begin{split}
 	&\frac{\alpha_1}{2}\|\lambda^{-m} \sqrt{G} \nabla V\|_{\Omega_{0,\tau}}^2\\
&\qquad\qquad \leq \left(\frac{(\tau + \eta)^2}{2} + \frac{2}{m}h_m(\tau, \eta)\right)
	\|\lambda^{-m}\sqrt{G} LV\|_{\Omega_{0, \tau}}^2 +
	\frac{2C_1}{m}h_m(\tau, \eta)\|\lambda^{-m}\sqrt{G}\nabla V\|_{\Omega_{0, \tau}}^2\\
	&\qquad\qquad\qquad + \!\!\!\dint{\supp{G}\times[0, \tau]}\!\!\!\lambda^{-2m}\frac{|\nabla G|^2}{G}
		\left(\frac{2C_1}{m}h_m(\tau, \eta)|\nabla V|^2 + n\frac{\alpha_2^2}{\alpha_1}|V|^2\right)\,d\mu\,dt.
\end{split}
\end{align*}
Choosing $T^{\prime}\leq T_2 $ and $\eta^{\prime} \leq \eta_2$ sufficiently
small, and  $m^{\prime} \in \mathbb{N}$ sufficiently large, to ensure that
\[
	\frac{2C_1}{m^{\prime}}h_{m^{\prime}}(T^{\prime}, \eta^{\prime}) 
	= C_1\left(\frac{1}{m^{\prime}} + 2(T^{\prime} + \eta^{\prime}) 
	+ \frac{2C^{\prime\prime}}{m^{\prime}}(T^{\prime} + \eta^{\prime})^2\right) < 
	\frac{\alpha_1}{4},
\]
we can absorb the second term on the right in the previous inequality into the left-hand side and obtain
\begin{align}\label{eq:pde2lower3}
\begin{split}
 	&\frac{\alpha_1}{4}\|\lambda^{-m} \sqrt{G} \nabla V\|_{\Omega_{0,\tau}}^2
 \leq \left(\frac{(\tau + \eta)^2}{2} + \frac{2}{m}h_m(\tau, \eta)\right)
	\|\lambda^{-m}\sqrt{G} LV\|_{\Omega_{0, \tau}}^2\\
	&\qquad\qquad+{ } \iint\limits_{\supp G \times[0,\tau]}\!\!\!\!\lambda^{-2m}\frac{|\nabla G|^2}{G}
		\left(\frac{2C_1}{m}h_m(\tau, \eta)|\nabla V|^2 + n\frac{\alpha_2^2}{\alpha_1}|V|^2\right)\,d\mu\,dt
\end{split}
\end{align}
for all $\tau \leq T^{\prime}$, $\eta \leq \eta^{\prime}$, and $m \geq m^{\prime}$.

Now, by \eqref{eq:pde1} again, we have
\begin{align*}
\begin{split}
	\frac{\alpha_1}{4}\|\lambda^{-(m+1)}\sqrt{G}V\|_{\Omega_{0, \tau}}^2&\leq
	\frac{\alpha_1C_1}{2m}\|\lambda^{-m}\sqrt{G}\nabla V\|^2_{\Omega_{0, \tau}}
	+ \frac{\alpha_1}{2m}\|\lambda^{-m}\sqrt{G}LV\|_{\Omega_{0, \tau}}^2\\
	&\quad 
		+ \frac{\alpha_1C_1}{2m}\!\!\!\dint{\supp{G}\times[0, \tau]}\!\!\!\lambda^{-2m}
		  \frac{|\nabla G|^2}{G}|\nabla V|^2\, d\mu\,dt,
\end{split}
\end{align*}
so choosing $m^{\prime\prime} \geq m^{\prime}$ large enough to ensure
$C_1/ m^{\prime\prime} \leq 1/4$, and summing the above inequality with \eqref{eq:pde2lower3},
we obtain
\begin{align}\label{eq:pde2fin}
\begin{split}
	&\frac{\alpha_1}{8}\|\lambda^{-m}\sqrt{G}\nabla V\|_{\Omega_{0, \tau}}^2
+\frac{\alpha_1}{4}\|\lambda^{-(m+1)}\sqrt{G}V\|^2_{\Omega_{0, \tau}}\\
&\qquad\leq  \left(\frac{\alpha_1 + 4 h_m(\tau, \eta)}{2m} + \frac{(\tau + \eta)^2}{2} \right)
	\|\lambda^{-m}\sqrt{G} LV\|_{\Omega_{0, \tau}}^2\\
&\qquad\qquad 	
 {}+ \dint{\supp{G}\times[0,\tau]}\!\!\!\lambda^{-2m}\frac{|\nabla G|^2}{G}
		\left(\frac{\alpha_1 +4C_1}{2m}h_m(\tau, \eta)|\nabla V|^2 + n\frac{\alpha_2^2}{\alpha_1}|V|^2\right)\,d\mu\,dt
\end{split}  
\end{align}
for all $m \geq m^{\prime\prime}$, $\tau \leq T^{\prime}$, $\eta \leq \eta^{\prime}$.
Finally, we estimate the last term above by
\begin{align*}
\begin{split}
 &\dint{\supp{G}\times[0,\tau]}\!\!\!\lambda^{-2m}\frac{|\nabla G|^2}{G}
		\left(\frac{\alpha_1 +4C_1}{2m}h_m(\tau, \eta)|\nabla V|^2 + n\frac{\alpha_2^2}{\alpha_1}|V|^2\right)\,d\mu\,dt\\
	&\;\leq \left(\frac{\alpha_1 + 4C_1}{2m}h_m(\tau, \eta) + n\frac{\alpha_2^2}{\alpha_1}(\tau + \eta)^2\right)
		 \!\!\!\dint{\supp{G}\times[0,\tau]}\!\!\!\frac{|\nabla G|^2}{G}
		\left(|\lambda^{-m}\nabla V|^2 + |\lambda^{-(m+1)}V|^2\right)\,d\mu\,dt.
\end{split}
\end{align*}
Thus we may assume that in \eqref{eq:pde2fin}, the coefficients of $\|\lambda^{-m}\sqrt{G}LV\|^2$ and the last integral
have the same basic structure.  
Multiplying both sides of \eqref{eq:pde2fin} by $4/\alpha_1$, and choosing $C_2$ sufficiently large,
we obtain \eqref{eq:pde2}
for $ m \geq m_1 \df m^{\prime\prime}$, $\tau \leq T_3 \df T^{\prime}$, and $\eta \leq \eta_3 \df \eta^{\prime}$.
\end{proof}

\subsection{Global estimates.}

In the proof of Theorem \ref{thm:bu}, we will need global versions of Lemmas \ref{lem:odecarleman}
and \ref{lem:pde2}.  As before, we work on the bundle $E= \bigoplus_{j=1}^N T^{k_j}_{l_j}(M)$.
If $M$ is compact, we can simply take $\Omega = M$, and $G\equiv 1$,
however, in general, we will need to impose further conditions on the growth of the sections
and additional controls on the metrics $g(t)$, and coordinate our application of the Lemmas 
of the preceding section with an appropriate family of 
cut-off functions $G = G_R$.

Specifically, we will assume below that the 
metrics $g(t)$ on $M$ are complete and satisfy
\begin{equation}\label{eq:rcbound}
		\Rc(g(t))\geq -Kg(t)
\end{equation}
for some $K \geq 0$, with time-derivatives $\pdt g = b$ satisfying
\begin{equation}\label{eq:bbound}
	|b|_{g(t)}^2 + |\nabla b|^2_{g(t)} \leq P_0
\end{equation} 
for $P_0 >  0$  on $M\times[0, T]$.  We will also assume the bounds on $\Lambda$ and its derivatives 
hold uniformly on $M\times [0, T]$:
\begin{equation}\label{eq:lambdabd}
	\alpha_1 g^{ij}(x, t)\leq \Lambda^{ij}(x,t)\leq \alpha_2 g^{ij}(x, t), 
	\qquad \left|\pd{\Lambda}{t}\right|^2_{g(t)} + |\nabla \Lambda|^2_{g(t)} \leq Q_0.
\end{equation}

The Ricci curvature lower bounds \eqref{eq:rcbound} 
imply, via the Bishop-Gromov volume comparison theorem, 
that there are constants $a^{\prime}$ and $A^{\prime}$ depending
only on $n$ and $K$ such that
\[
	\vol(B_{g(t)}(x_0, R)) \leq A^{\prime}e^{a^{\prime} R}.
\]
The uniform bounds on $b(t)$ imply that the metrics $g(t)$ are uniformly equivalent on $[0, T_0]$,
in fact,
\[
	e^{-P_0(t_2- t_1)}g(x, t_1) \leq g(x, t_2) \leq e^{P_0(t_2-t_1)}g(x, t_1),
\]
and the bounds on $\nabla b(t)$ imply those on $\pdt\nabla$, in view of \eqref{eq:evgamma}.
With the uniform equivalence, we have
\[
	\frac{1}{C_6}d_{g(t)}(x, x_0) \leq r(x) \df d_{g(0)}(x, x_0) \leq C_6 d_{g(t)}(x, x_0)
\]
for $C_6 = C_6(P_0T)$, and thus that
\[
	\Omega_R \df \left\{\,r(x) < R\, \right\} \subset B_{g(t)}(x_0, C_6R).
\]
Consequently, there exists $a^{\prime\prime} = a^{\prime\prime}(K, P_0T)$ such that
\begin{equation}\label{eq:volgrwth2}
	\vol_{g(t)}(\Omega_R) \leq A^{\prime} e^{a^{\prime\prime}R}	
\end{equation}
for all $R$ and $t\in [0, T]$. 

Before we formulate our global estimate, we define the function
\[
	\varphi_a(x) \df \varphi_{a, x_0, 0}(x) \df  e^{-ar(x)}
\]
for a parameter $a\geq 0$ to be determined later. 
Off of the $g(0)$-cut locus of $x_0$, we have $|\nabla r|^2_{g(0)}\equiv 1$, and
in view of the uniform equivalence, we have $|\nabla r|_{g(t)}^2 \leq C_7$
for some $C_7 = C_7(P_0T)$.  Thus, $d\mu_{g(t)}$-a.~e.~, we have 
\[
	|\nabla\varphi_a|^2_{g(t)} \leq C_7 a^2 \varphi_a^2. 
\]

\begin{proposition}\label{prop:global}
	For any $\alpha_1$, $\alpha_2$, $a_0$, $P_0$, $Q_0 > 0$, and $K\geq 0$,
	there exist positive numbers $a_1$, $\eta_4$, $m_2$ and $T_4 \leq 1$, depending only on this data
	such that if $\tau \leq T_4$ and 
	$g(t)$ and $\Lambda(t)$ are as above, satisfying the bounds
	(\ref{eq:rcbound}) -- (\ref{eq:lambdabd}) on $M\times[0, \tau]$,
	then
\begin{enumerate}
\item If $V \in \mathcal{V}_{M}^{0, \tau}(E)$ satisfies
	$|V|_{g(t)}^2 \leq e^{a_0d_{g(t)}(x_0, x)}$, then 
\begin{equation}\label{eq:globalode}
	\liminf_{R\to\infty}\left\|\lambda^{-m}\sqrt{\varphi_{a_1}}\pd{V}{t}\right\|_{\Omega_{2R}\times [0, \tau]}^2 
		\geq m^2 \|\lambda^{-(m+1)}\sqrt{\varphi_{a_1}}V\|_{M_{0, \tau}}^2
\end{equation}	
for all $m \geq m_2$, $\eta \leq \eta_4$.
\item If $V \in \mathcal{V}_{M}^{0, \tau}(E)$ satisfies
	$|V|_{g(t)}^2 + |\nabla V|_{g(t)}^2 \leq e^{a_0d_{g(t)}(x_0, x)}$, then 
\begin{align}\label{eq:globalpde2}
\begin{split}
	&\liminf_{R\to\infty}\rho_m(\tau, \eta)\|\lambda^{-m}\sqrt{\varphi_{a_1}}LV\|_{\Omega_{2R}\times[0, \tau]}^2
\phantom{\geq \frac{1}{4}\|\lambda^{-(m+1)}\sqrt{\varphi_{a_1}}V\|_{M_{0, \tau}}^2}\\
&\qquad\qquad\geq
\frac{3}{4}\|\lambda^{-(m+1)}\sqrt{\varphi_{a_1}}V\|_{M_{0, \tau}}^2\phantom{\geq} {}+\frac{1}{4}\|\lambda^{-m}\sqrt{\varphi_{a_1}}\nabla V\|_{M_{0, \tau}}^2
\end{split}
\end{align}
for all $m \geq m_2$, $\eta \leq \eta_4$, where $\rho_m$ is as defined in Lemma \ref{lem:pde2}.
When $M$ is compact, the estimates hold with $a_1 = 0$, $\varphi_{a_1} \equiv 1$.
\end{enumerate}
\end{proposition}

\begin{proof}
Let $0 < \tau \leq 1$.
First, we construct the cut-off function. We choose  a monotone-decreasing
$\eta\in C^{\infty}(\mathbb{R}, [0, 1])$ with 
\[
	\eta(s)  = \left\{
		    \begin{array}{cr} 0 & \mbox{if}\quad s \geq 2\\
				      1 & \mbox{if}\quad s \leq 1,
		   \end{array}
		   \right.
\]
and $(\eta^{\prime})^2 \leq \beta \eta$.  Then, with $r(x) = d_{g(0)}(x, x_0)$ as before, we define the function
\[
	G_R(x)\df G_{x_0, R, a}(x) \df \eta\left(\frac{r(x)}{R}\right)\varphi_{a}(x)
\]
which is supported in $\Omega_{2R}$ and is smooth off of $\operatorname{cut}_{g(0)}(x_0)\cap \Omega_{2R}$,
where
\[
	\nabla G_R = \left(\frac{\eta^{\prime}}{R} - a\eta\right)\varphi_a\nabla r. 
\]
As we have observed above, there is a constant $C_7 = C_7(P_0)$ (note $\tau \leq 1$)
such that $|\nabla r|^2_{g(\tau)} \leq C_7$, and it follows that $d\mu_{g(t)}$-a.~e.~ on $\Omega_{2R}$,
we have
\begin{equation}\label{eq:ggrad}
	\frac{|\nabla G_R|^2}{G_R}(x) \leq C_8\left\{\begin{array}{lc}
						 \left(\frac{\beta}{R^2} +  a^2\right)\varphi_a
						 &\mbox{on}\quad A(R, 2R),\\
						 a^2\varphi_a &\mbox{on}\quad\Omega_R
						\end{array}\right. 	
\end{equation}
for some $C_8 = C(P_0)$, where $A(R, 2R) \df \Omega_{2R}\setminus\Omega_R$.
Also, as above, we know there exist positive constants $A_2$, $a_2=a_2(K, P_0)$ such that
\begin{equation}\label{eq:volgrwth}
	\vol_{g(t)}(\Omega_{\rho}) \leq A_2 e^{a_2 \rho}
\end{equation}
for all $\rho > 0$ and $t\in [0, \tau]$, and the uniform equivalence implies that there exist constants $A_3$, $a_3 = a_3(a_0, P_0)$,
such that, under either the growth assumptions of (1) or (2), we have
\begin{equation}\label{eq:secgrwth}
	|V|_{g(t)}^2(x, t) \leq A_3 e^{a_3 r(x)}, \quad\mbox{or} \quad 
	|V|_{g(t)}^2 + |\nabla V|_{g(t)}^2 \leq A_3e^{a_3 r(x)}. 
\end{equation}
Thus, in the definition of $G_R$, we will take
$a = a_1 \df 4 \times \max\{a_2, a_3\}$.

To prove \eqref{eq:globalode}, we apply \eqref{eq:volgrwth} and \eqref{eq:secgrwth} and 
observe that, for all $R \geq 1$, we have, for our choice of $a_1$,
\begin{align*}
	\|\lambda^{-(m+1)}\sqrt{\varphi_{a_1}} V\|_{\Omega_R\times[0, T]}^2 
	&\leq \|\lambda^{-(m+1)}\sqrt{G_R} V\|_{\Omega_{2R}\times[0, \tau]}\\
	 &\leq\frac{A_3}{\eta^{2m+2}}\!\!\!\dint{\Omega_{2R}\times[0, \tau]}\!\!\!e^{-(a_1 - a_3)r}\,d\mu_{g(t)}\\
	 &\leq\frac{A_3}{\eta^{2m+2}}\sum_{l=0}^{\lceil 2R \rceil - 1}\dint{A(l, l + 1)\times [0, \tau]}\!\!\!
		e^{-(a_1 - a_3)r}\,d\mu_{g(t)}\\
	 &\leq A_2 A_3 \frac{\tau}{\eta^{2m+2}}\left(\frac{e^{a_2}}{1 - e^{-(a_1 - a_2 - a_3)}}\right)  
\end{align*}
The sets $\Omega_R$ exhaust $M$, and 
thus, taking $T_1 = T_1(P_1)$ and $\eta_1 = \eta_1(P_1)$ as in Lemma \ref{lem:odecarleman},
we have \eqref{eq:globalode} for all $\tau \leq T^{\prime} \df \min\{T_1, 1\}$ and 
$\eta \leq  \eta_1$ upon sending $R\to\infty$ along any subsequence.  

To prove \eqref{eq:globalpde2}, we apply Lemma \ref{lem:pde2} to $V$ on 
$\Omega_{2R}\times [0, \tau]$  
with data $\alpha_1$, $\alpha_2$, $P_0$, $Q_0$, and $G = G_R$.
Up to a multiplicative factor, the integral in the final term of equation \eqref{eq:pde2} has the form
$I(R) = I_1(R) + I_2(R)$, where
\[
	I_2(R) \df \dint{A(R, 2R)\times[0, \tau]}\!\!\!\frac{|\nabla G_R|^2}{G_R}
			(|\lambda^{-(m+1)}V|^2 + |\lambda^{-m}\nabla V|^2)\,d\mu\,dt.
\]
By \eqref{eq:ggrad} -- \eqref{eq:secgrwth},
$I_2$ satisfies
\begin{align*}
	I_2(R) &\leq \tau A_2A_3C_8\left(\frac{\eta^2 + 1}{\eta^{2m+2}}\right)
	\left(\frac{\beta}{R^2} +  a_1^2\right)e^{-(a_1 - 2a_2 - a_3)R}
\end{align*}
for all $R$ and $\eta > 0$, and so we have $\lim_{R\to\infty}I_2(R) = 0$.
Thus, again using \eqref{eq:ggrad}, we have
\begin{align*}
	\lim_{R\to\infty} I(R) &= \lim_{R\to\infty} \dint{\Omega_{R}\times[0, \tau]}\!\!\!\frac{|\nabla G_R|^2}{G_R}
			(|\lambda^{-(m+1)}V|^2 + |\lambda^{-m}\nabla V|^2)\,d\mu\,dt\\
			       &\leq C_8 a_1^2 \left(\|\lambda^{-(m+1)}\sqrt{\varphi_{a_1}} V\|^2_{M_{0, \tau}}
					+ \|\lambda^{-m}\sqrt{\varphi_{a_1}}\nabla V\|^2_{M_{0, \tau}}\right),	    	 
\end{align*}
where the finiteness of the terms on the right-hand are easily established using
\eqref{eq:volgrwth} and \eqref{eq:secgrwth} and the argument for \eqref{eq:globalode}.
Now let $T_3$, $\eta_3$, and $m_1$ be as in Lemma \ref{lem:pde2}, and choose
$T^{\prime\prime} \leq \min\{T_3, 1\}$, $\eta^{\prime}\leq \eta_3 $, and $m^{\prime}\geq m_1$ such that 
\[
	\rho_{m^{\prime}}(T^{\prime\prime}, \eta^{\prime}) C_8 a_1^2 < \frac{1}{4}.
\]
Then, we may estimate the limit of the right-hand side of \eqref{eq:pde2} along any subsequence $R_j\to \infty$
from below as
\begin{align*}
&\lim_{j\to\infty} \bigg\{\frac{1}{2}\|\lambda^{-m}\sqrt{G_{R_j}}\nabla V\|_{\Omega_{2R_j}\times[0, \tau]}^2
		+ \|\lambda^{-(m+1)}\sqrt{G_{R_j}}V\|^2_{\Omega_{2R_j}\times[0, \tau]}\\
&\qquad
		{}- \rho_m(\tau, \eta)\!\!\!\dint{\Omega_{2R_j}\times[0, \tau]}\!\!\!
		\frac{|\nabla G_{R_j}|^2}{G_{R_j}}\left(|\lambda^{-m}\nabla V|^2 +
				|\lambda^{-(m+1)} V|^2\right)\,d\mu\,dt 
\bigg\}\\
&\qquad\qquad\geq \frac{3}{4}\|\lambda^{-(m+1)}\sqrt{\varphi_{a_1}}V\|^2_{M_{0, \tau}} 
		+ \frac{1}{4}\|\lambda^{-m}\sqrt{\varphi_{a_1}}\nabla V\|^2_{M_{0, \tau}}, 
\end{align*}
for all $0 < \tau \leq T^{\prime\prime}$, $0 < \eta \leq \eta^{\prime}$, and $m \geq m^{\prime}$,
which implies \eqref{eq:globalpde2}.   This completes the proof, taking $T_4 \df \min\{T^{\prime}, T^{\prime\prime}\}$,
$\eta_4 \df \min\{\eta^{\prime}, \eta^{\prime}\}$, and $m_2 \df m^{\prime}$.
\end{proof}

${}$

\subsection{A backwards-uniqueness theorem for a coupled system of differential inequalities.}

With Proposition \ref{prop:global} in hand, we turn to the proof of the main result of the section,
applying the lower estimates \eqref{eq:globalode}, \eqref{eq:globalpde2} in tandem with the matching
upper inequalities \eqref{eq:odeineq} and \eqref{eq:pdeineq} of our PDE-ODE system.  The mechanism 
in the proof is based on that in \cite{LP}.

\begin{proof}[Proof of Theorem \ref{thm:bu}]
The result will be a consequence of the following claim: 

\emph{For any (positive) choice of the data 
$\alpha_1$, $\alpha_2$, $a$, $A$, $C$, $K$ $P$, $Q$, there exists $T_5 = T_5(\alpha_1, \alpha_2, a, A, C, K,  P, Q) > 0$
such that if $0 < \tau \leq T_5$ and $g(t)$ is a complete family of metrics, satisfying, along with 
$\Lambda(t)$, $X(t)$, and $Y(t)$, the assumptions
(\ref{eq:pq}) -- (\ref{eq:odeineq}) on $M\times [0, \tau]$, then $X(\tau) \equiv 0$, $Y(\tau)\equiv 0$
implies $X(t) \equiv 0$, $Y(t)\equiv 0$ on $[0, \tau]$.}

Indeed, if assumptions (\ref{eq:pq}) - (\ref{eq:odeineq}) are met on $M\times[0, T]$ with $T > T_5$, and $X(T) \equiv 0$, 
$Y(T)\equiv 0$, then one may apply the claim successively to $X_i(t)\df X(t + (T- iT_5)_+)$, and $Y_i(t) \df Y(t + (T-iT_5)_+)$ on 
$[0, \min\{T - T_5, T- (i-1)T_5\}]$, for $i = 1, 2, \ldots, \lceil T/T_5\rceil - 1$.

It remains to prove the claim. We let $a_1$, $m_2$, $T_4$, and $\eta_4$ be the constants supplied by 
Proposition \ref{prop:global} and choose $T_5 \leq T_4$, $\eta_5 \leq \eta_4$, and $m_3 \leq m_2$ 
further depending on $C$ so that
\begin{equation}\label{eq:rhoc}
	C\max\left\{\rho_{m_3}(T_5, \eta_5), \frac{1}{m_3^2}\right\} \leq \frac{1}{8},	
\end{equation}
and
\begin{equation}\label{eq:lastconst}
	T_5 + \eta_5 \leq 1.
\end{equation}
Then we fix $0 < \tau \leq T_5$, $0 < \eta \leq \eta_5$, $m \geq m_3$.

Selecting arbitrary  $0 < t_1 < t_2 < \tau$, we choose a non-decreasing 
$\xi\in C^{\infty}(\mathbb{R}, [0, 1])$ such that
\[
\left\{\begin{array}{lc}
	\xi(t) &= 1\quad\mbox{for}\quad t > t_2\phantom{.}\\
	\xi(t) &= 0\quad\mbox{for}\quad t < t_1.
\end{array}\right.
\]
Then $\tilde{X}(x,t) \df \xi(t)X(x, t)$ and $\tilde{Y}(x, t)\df \xi(t)Y(x, t)$ belong to the classes 
$\mathcal{V}_1\df \mathcal{V}_M^{0, \tau}(\mathcal{X})$, and
$\mathcal{V}_2\df\mathcal{V}_M^{0, \tau}(\mathcal{Y})$,
respectively. 

Now,
\[
    |L\tilde{X}|^2 \leq 2\xi^2|LX|^2 + 2(\xi^{\prime})^2|X|^2, \quad\mbox{and}
	\quad \left|\pd{\tilde{Y}}{t}\right|^2 \leq 2 \xi^2\left|\pd{Y}{t}\right|^2 + 2(\xi^{\prime})^2|Y|^2.
\]
Thus, with the growth assumptions \eqref{eq:burcgrwth}, the 
Ricci curvature lower bound \eqref{eq:bugrwth}, 
and inequalities \eqref{eq:pdeineq} and \eqref{eq:odeineq},
we have, as in the proof of Proposition \ref{prop:global}, that
\[
\lim_{R\to\infty}\|\lambda^{-m}\sqrt{\varphi_{a_1}}LX\|_{\Omega_{2R}\times[0, \tau]}^2 
= \|\lambda^{-m}\sqrt{\varphi_{a_1}}LX\|_{M_{0, \tau}}^2 < \infty,	
\]
and
\[
\lim_{R\to\infty}\left\|\lambda^{-m}\sqrt{\varphi_{a_1}}\pd{Y}{t}\right\|_{\Omega_{2R}\times[0, \tau]}^2 
= \left\|\lambda^{-m}\sqrt{\varphi_{a_1}}\pd{Y}{t}\right\|_{M_{0, \tau}}^2 < \infty.
\]
Then,
\begin{align*}
	&\frac{1}{m^2}\left\|\lambda^{-m}\sqrt{\varphi_{a_1}}\pd{\tilde{Y}}{t}
	\right\|^2_{M_{0, \tau}} 
	=\frac{1}{m^2}\left(\left\|\lambda^{-m}\sqrt{\varphi_{a_1}}\pd{\tilde{Y}}{t}\right\|^2_{M_{t_1, t_2}} 
	+ \left\|\lambda^{-m}\sqrt{\varphi_{a_1}}\pd{Y}{t}\right\|^2_{M_{t_2, \tau}}\right)\\
	&\qquad\leq\frac{1}{m^2} \left\|\lambda^{-m}\sqrt{\varphi_{a_1}}\pd{\tilde{Y}}{t}
	\right\|^2_{M_{t_1, t_2}}\\	
	&\qquad\phantom{\leq}{}+ \frac{1}{8}\left(\|\lambda^{-m}\sqrt{\varphi_{a_1}}X\|_{M_{t_2, \tau}}^2
	+ \|\lambda^{-m}\sqrt{\varphi_{a_1}}\nabla X\|_{M_{t_2, \tau}}^2
	+ \|\lambda^{-m}\sqrt{\varphi_{a_1}} Y\|_{M_{t_2, \tau}}^2\right)
\end{align*}
and 
\begin{align*}
	&\rho_m(\tau, \eta)\|\lambda^{-m}\sqrt{\varphi_{a_1}}L \tilde{X}\|^2_{M_{0, \tau}} \\
&\qquad= \rho_m(\tau, \eta)\left(\|\lambda^{-m}\sqrt{\varphi_{a_1}}L \tilde{X}\|^2_{M_{t_1,t_2 }}
		+ \|\lambda^{-m}\sqrt{\varphi_{a_1}}L X\|^2_{M_{t_2, \tau}}\right)\\
&\qquad\leq\rho_m(\tau, \eta)\|\lambda^{-m}\sqrt{\varphi_{a_1}}L \tilde{X}\|^2_{M_{t_1,t_2 }}\\
&\qquad\qquad {}+ \frac{1}{8}\left(\|\lambda^{-m}\sqrt{\varphi_{a_1}}X\|_{M_{t_2, \tau}}^2
	+ \|\lambda^{-m}\sqrt{\varphi_{a_1}}\nabla X\|_{M_{t_2, \tau}}^2
	+ \|\lambda^{-m}\sqrt{\varphi_{a_1}} Y\|_{M_{t_2, \tau}}^2\right),
\end{align*}
using \eqref{eq:pdeineq} and \eqref{eq:odeineq} together with \eqref{eq:rhoc}.
On the other hand, by \eqref{eq:globalode} and \eqref{eq:globalpde2}, we have
\begin{align*}
	\|\lambda^{-(m+1)}\sqrt{\varphi_{a_1}} Y\|^2_{M_{t_2, \tau}}
	&\leq\|\lambda^{-(m+1)}\sqrt{\varphi_{a_1}}\tilde{Y}\|^2_{M_{0, \tau}}\\
	&\leq\frac{1}{m^2}\left\|\lambda^{-m}\sqrt{\varphi_{a_1}}\pd{\tilde{Y}}{t}\right\|^2_{M_{0, \tau}}
\end{align*}
and 
\begin{align*}
	&\frac{1}{4}\|\lambda^{-m}\sqrt{\varphi_{a_1}}\nabla X\|_{M_{t_2, \tau}}^2
	+ \frac{3}{4}\|\lambda^{-(m+1)}\sqrt{\varphi_{a_1}} X\|_{M_{t_2, \tau}}^2\\
	&\qquad\qquad\qquad\leq \frac{1}{4}\|\lambda^{-m}\sqrt{\varphi_{a_1}}\nabla \tilde{X}\|_{M_{0, \tau}}^2
	+ \frac{3}{4}\|\lambda^{-(m+1)}\sqrt{\varphi_{a_1}}\tilde{X}\|_{M_{0, \tau}}^2\\
	&\qquad\qquad\qquad \leq \rho_m(\tau, \eta)\|\lambda^{-m}\sqrt{\varphi_{a_1}}L\tilde{X}\|_{M_{0, \tau}}^2.
\end{align*}
Since, by \eqref{eq:lastconst},
\begin{align*}
	|\lambda^{-m} X|^2 + |\lambda^{-m}Y|^2 &\leq (\tau + \eta)^2 (|\lambda^{-(m+1)} X|^2 + |\lambda^{-(m+1)}Y|^2)\\
			&\leq|\lambda^{-(m+1)} X|^2 + |\lambda^{-(m+1)}Y|^2,
\end{align*}
we can combine the corresponding parts of the previous four sets of inequalities to obtain
\begin{align*}
\frac{7}{8}\|\lambda^{-(m+1)}\sqrt{\varphi_{a_1}}Y\|_{M_{t_2, \tau}}^2 
	&\leq\frac{1}{m^2} \left\|\lambda^{-m}\sqrt{\varphi_{a_1}}\pd{\tilde{Y}}{t}\right\|^2_{M_{t_1, t_2}}\\
&\phantom{\leq}\qquad{}+ \frac{1}{8}\left(\|\lambda^{-(m+1)}\sqrt{\varphi_{a_1}}X\|_{M_{t_2, \tau}}^2
			+ \|\lambda^{-m}\sqrt{\varphi_{a_1}}\nabla X\|_{M_{t_2, \tau}}^2\right) 
\end{align*}
and
\begin{align*}
\frac{5}{8}\|\lambda^{-(m+1)}\sqrt{\varphi_{a_1}} X\|_{M_{t_2, \tau}}^2 
	&+\frac{1}{8}\|\lambda^{-m}\sqrt{\varphi_{a_1}}\nabla X\|_{M_{t_2, \tau}}^2\\
&\leq\rho_m(\tau, \eta)\|\lambda^{-m}\sqrt{\varphi_{a_1}}L \tilde{X}\|^2_{M_{t_1,t_2 }}
+ \frac{1}{8}\|\lambda^{-(m+1)}\sqrt{\varphi_{a_1}} Y\|_{M_{t_2, \tau}}^2.\qquad\qquad\qquad\qquad\qquad
\end{align*}
Adding these two inequalities implies, then, that
\begin{align*}
&\frac{1}{2}\left(\|\lambda^{-(m+1)}\sqrt{\varphi_{a_1}}X\|^2_{M_{t_2, \tau}}
		+ \|\lambda^{-(m+1)}\sqrt{\varphi_{a_1}}Y\|^2_{M_{t_2, \tau}}\right)\\
	&\qquad\qquad\leq \rho_m(\tau, \eta)\|\lambda^{-m}\sqrt{\varphi_{a_1}}L\tilde{X}\|_{M_{t_1, t_2}}^2
	+ \frac{1}{m^2}\left\|\lambda^{-m}\sqrt{\varphi_{a_1}}\pd{\tilde Y}{t}\right\|_{M_{t_1, t_2}}^2.
\end{align*}	
Using the notation $W \df X\oplus Y$, $\tilde{W} \df \xi W$, $\mathcal{L}(W) \df LX \oplus \pdt Y$, and 
\[
	\sigma_m(\tau, \eta) \df 2\max\{\rho_m(\tau, \eta), 1/m^2 \},
\] 
we can write the above more economically as
\begin{equation}\label{eq:pen}
	\|\lambda^{-(m+1)}\sqrt{\varphi_{a_1}}W\|^2_{M_{t_2, \tau}} 
		\leq \sigma_m(\tau, \eta)\|\lambda^{-m}\sqrt{\varphi_{a_1}}\mathcal{L}\tilde{W}\|^2_{M_{t_1, t_2}}. 	
\end{equation}

Choosing now $t_2 < t_3 < \tau$, it follows from \eqref{eq:pen} that
\begin{equation}\label{eq:last}
	\|\lambda^{-m}\sqrt{\varphi_{a_1}}W\|^2_{M_{t_3, \tau}} 
		\leq \sigma_m(\tau, \eta)\|\lambda^{-m}\sqrt{\varphi_{a_1}}\mathcal{L}\tilde{W}\|^2_{M_{t_1, t_2}}. 	
\end{equation}
Since $\lambda^{-m}(t) \leq (\tau- t_2 + \eta)^{-m}$ on $[t_1, t_2]$, and
	  $\lambda^{-(m+1)}(t) \geq (\tau - t_3 + \eta)^{-(m+1)}$ on $[t_3, \tau]$, 
	\eqref{eq:last} implies
\begin{align*}
	\|\sqrt{\varphi_{a_1}}W\|^2_{M_{t_3, \tau}}&\leq \sigma_m(\tau, \eta)
			\frac{(\tau - t_3 + \eta)^{2(m+1)}}{(\tau- t_2 + \eta)^{2m}}
			\|\sqrt{\varphi_{a_1}}\mathcal{L}\tilde{W}\|_{M_{t_1, t_2}}^2\\
			&\leq \sigma_m(\tau, \eta)
			\left(\frac{\tau - t_3 + \eta}{\tau- t_2 + \eta}\right)^{2m}
			\|\sqrt{\varphi_{a_1}}\mathcal{L}\tilde{W}\|_{M_{t_1, t_2}}^2,
\end{align*}
for all $m \geq m_0$ (to obtain the second inequality, we have again used that
$\tau- t_3 + \eta \leq 1$ in view of \eqref{eq:lastconst}).
But $\sigma_m(\tau, \eta)$ remains bounded as $m$ increases, and upon sending $m\to \infty$,
we conclude
\[
	\|\sqrt{\varphi_a}W\|^2_{M_{t_3, \tau}}  = 0,
\]
and hence that $X\equiv 0$ and $Y\equiv 0$ on $[t_3, \tau]$.  Since $t_3$ can be chosen arbitrarily close
to $t_2$, and the pair $t_1$, $t_2$ arbitrarily close to $0$,  we arrive at the same conclusion on $[0, \tau]$.
This completes the proof.
\end{proof}

\section{Backwards uniqueness for the Ricci flow}

Theorem \ref{thm:rfbu} now follows easily from Proposition \ref{prop:diffineq} and Theorem \ref{thm:bu}.

\begin{proof}[Proof of Theorem \ref{thm:rfbu}]
It suffices to show that $g(t) \equiv \gt(t)$ on $[\delta, T]$
for any $0 < \delta < T$.  In view of Proposition \ref{prop:diffineq}, we need only to check
that $\ve{X}$, $\nabla \ve{X}$, and $\ve{Y}$ satisfy the required growth condition \ref{eq:bugrwth},
and it was shown in the proof of that proposition that the components of $\ve{X}$ and $\ve{Y}$,
and the quantities $|\nabla^{(m)} \Rmt_{g(t)}|$, $|\delt^{(m)}\Rmt|_{g(t)}$ and $|\nabla^{(m)}A|_{g(t)}$
are in fact bounded on $M\times [\delta, T]$.  Thus it remains only to show that $|\nabla T|_{g(t)}$ and $|\nabla U|_{g(t)}$
(and hence $|\nabla\ve{X}|_{g(t)}$) are bounded.
But from an inductive application of the identity
\[
	\nabla^{(m)}\delt^{(l)}\Rmt = \nabla^{(m-1)}\delt^{(l+1)}\Rmt 
	+ \sum_{i = 0}^{m-1}\nabla^{(i)}A\ast\nabla^{(m - 1 - i)}\delt^{(l)}\Rmt, 
\]
we obtain bounds on ``mixed'' derivatives of the form $\nabla^{(m)}\delt^{(l)}\Rmt$ for all $m$ and $l\geq 0$,
from which it follows, in particular, that $\nabla T = \nabla \Rm - \nabla \Rmt$ and 
$\nabla U = \nabla\nabla \Rm - \nabla\delt\Rmt$ are bounded.

The above remarks yield a constant
$C^{\prime} = C^{\prime}(\delta, K, \tilde{K}, T)$ such that $\ve{X} = T\oplus U$
and $\ve{Y} = h \oplus A \oplus B$ satisfy
\[
	|\ve{X}|_{g(t)}^2 + |\nabla\ve{X}|_{g(t)}^2 + |\ve{Y}|_{g(t)}^2 \leq C^{\prime}
\]	  
on $M\times[0, T]$.  Applying then Theorem \ref{thm:bu} to $\ve{X}$ and $\ve{Y}$ via Proposition \ref{prop:diffineq}, 
with metrics $g(t)$ and $\Lambda^{ij}(t) = g^{ij}(t)$
on $M\times [\delta, T]$, we conclude that $\ve{X} \equiv \ve{0}$, $\ve{Y}\equiv \ve{0}$,
and hence that $h(t) = g(t) -\gt(t) \equiv 0$, for $t\in [\delta, T]$.
\end{proof}

Theorem \ref{thm:isom} is now simply a consequence of the diffeomorphism invariance
of the Ricci tensor.

\begin{proof}[Proof of Theorem \ref{thm:isom}]
Suppose that $g(t)$ is a solution to \eqref{eq:rf} on $[0, T]$ with uniformly bounded curvature,
and $\phi\in \operatorname{Isom}(g(T))$.
Then $\gt(t) \df \phi^*(g(t))$ is a solution to \eqref{eq:rf} on $[0, T]$, since
\[
	\pdt \gt(t)  = \phi^*\left(-2\Rc(g(t))\right) = -2\Rc(\phi^*(g(t))) = -2\Rc(\gt(t)),
\]
and has uniformly bounded curvature. Since $\gt(T) = \phi^*(g(T)) = g(T)$, it follows from Theorem \ref{thm:rfbu}
that $g(t) = \gt(t) = \phi^*(g(t))$ for all $t \in [0, T]$, that is, $\phi \in \operatorname{Isom}(g(t))$
for all $t\in [0, T]$. 
\end{proof}

\begin{remark}
	It is interesting to note that the property of backwards-uniqueness is actually equivalent
	to the non-expansion of the isometry group.  Indeed, given two solutions $g(t)$, $\gt(t)$
	to \eqref{eq:rf} on $M\times [0, T]$, one may form the solution $h(t) = g(t)\oplus \gt(t)$
	on $M\times M\times [0, T]$.  Then, if $g(T) = \gt(T)$, 
	the map $\Phi: M\times M \to M\times M$ defined by $\Phi(x, y) = (y, x)$ is an isometry
	at $t= T$: 
\[
	\Phi^*(h(T)) = \gt(T)\oplus g(T) = g(T) \oplus \gt(T) = h(T).
\]
If $\operatorname{Isom}(h(T)) \subset \operatorname{Isom}(h(t))$ for all $t$, then we have
$\Phi^*(h(t)) = h(t)$, i.e., $g(t) = \gb(t)$ for all $t$.  In fact, this shows that the question of backward-uniqueness is equivalent
to the impossibility of a solution acquiring an isometry of order two 
(or, by a similar construction, of any finite order) within the lifetime of the solution.
\end{remark}

For the proof of Theorem \ref{thm:soliton}, we will need the following auxiliary lemma.

\begin{lemma}\label{lem:complete}
Suppose $g$ is a complete metric on $M$ with $|\Rc(g)|\leq K$ satisfying
\[
	\Rc(g) + \mathcal{L}_X g + \frac{\lambda}{2} g = 0
\]
for $X\in C^{\infty}(TM)$ and $\lambda\in \mathbb{R}$.  Then the vector field $X$ is complete.
\end{lemma} 
\begin{proof}
We have
\[
	\nabla_i X_j + \nabla_j X_i = S_{ij}
\]
for a bounded tensor $S$.  Fix $p_0\in M$, and let $\sigma: (a, b) \to M$ be the maximal solution
to the initial value problem
\[
	\left\{\begin{array}{lcl}
		\dot{\sigma}(\tau) &=& X(\sigma(\tau))\\
		\sigma(0) &=& p_0
	\end{array}\right.
\]
We claim $b = \infty$. 

Along $\sigma$, one has
\begin{align*}
	\frac{d}{d\tau}|X|^2(\sigma(\tau)) &= 2\langle \nabla_X X, X\rangle(\sigma(\tau)) 
		=  S(\sigma(\tau))(X, X)\leq C|X|^2(\sigma(\tau))
\end{align*}
for $\tau \in (a, b)$.
Thus $|X|^2(\sigma(\tau)) \leq e^{C\tau}|X|^2(p_0)$
for $0 \leq \tau < b$. But then
\begin{align*}
	d(p_0, \sigma(\tau)) &\leq \operatorname{length}(\sigma([0, \tau]))= \int_0^{\tau}|X|(\sigma(t))\,dt\\
	&\leq \frac{2}{C}|X(p_0)|(e^{C\tau/2} - 1)
\end{align*}
for $\tau$ in this range, and so, if $b < \infty$, we have 
$\limsup_{\tau\to b}d(p_0, \sigma(\tau)) < \infty$, 
contradicting the maximality of the interval $(a, b)$.
Similarly, one concludes $a = -\infty$.
\end{proof}

\begin{remark}
	If $X = \nabla f$, i.e., $g$ is a \emph{gradient} Ricci soliton, then Zhang \cite{Z} has shown
	that the completeness of the vector field $\nabla f$ follows from the completeness
	of the metric $g$ without the assumption of a curvature bound.
\end{remark}

With Lemma \ref{lem:complete}, from any $\bar{g}$, $X$ and $\lambda$ satisfying \eqref{eq:soliton} we may
construct in canonical fashion a complete self-similar solution $\gt(t)$ to the Ricci flow
with $\gt(T) = \bar{g}$.  By Theorem \ref{thm:rfbu}, this must agree identically with the solution $g(t)$
of Theorem \ref{thm:soliton}.
  
\begin{proof}[Proof of Theorem \ref{thm:soliton}]
If $\bar{g}$ is flat, then $\gt(t)\df \bar{g}$ is evidently a solution to \eqref{eq:rf} for all $t$, and 
$g(t) \equiv \gt(t)$ by Theorem \ref{thm:rfbu}. Thus we can take $\phi_t \equiv Id$ and $c \equiv 1$.

Suppose then that $\bar{g}$ is not flat. As in \cite{CK},
we define $c(t) \df \lambda(t- T) + 1$, and
take $\varphi_t:M\to M$ to be the family of diffeomorphisms defined by
\[
	\pdt\varphi_{t}(x) = 2\tilde{X}(x, t) \df \frac{2}{c(t)}X(\varphi_t(x)), \quad \varphi_T = Id.
\]
Since $g(T)$ has bounded Ricci curvature, in view of Lemma \ref{lem:complete}, $\varphi_t$ is defined
for $t\in \mathcal{I}\df \{\,c(t) > 0\,\}$.
The metrics $\gt(t)\df c(t)\varphi_{t}^*\bar{g}$, then, are complete and satisfy
\[
	\pdt \gt = -2\Rc(\gt), \quad\mbox{ and }\quad \Rc(\gt) +\mathcal{L}_{\varphi_{t}^*(\tilde{X})}\gt + \frac{1}{2c}\gt = 0 
\]
on $M\times \mathcal{I}$.
Moreover, the curvature of $\gt$ is uniformly bounded on any proper closed subinterval of $\mathcal{I}$.

Since $\gt(T) = \bar{g} = g(T)$, we apply Theorem \ref{thm:rfbu} to any subinterval $[t_1, T]$
of $\mathcal{I}\cap [0, T]$ to conclude $g(t) \equiv \gt(t)$ for all 
$t\in \mathcal{I}\cap [0, T]$. We claim that we must have $[0, T] \subset \mathcal{I}$.
This is clear if $\lambda \leq 0$,  as then $(-\infty, T] \subset \mathcal{I}$.
Otherwise, if $\lambda > 0$ and  
$0 \leq a\df T - 1/\lambda =  \inf \mathcal{I}$, we claim that 
$g(t)$ cannot have bounded curvature on $[0, T]$.  Since $\bar{g}$ is not flat, we can find $x_0 \in M$
with $\Rmb(x_0) \ne  0$.  But then, since $g(t) = c(t)\varphi_{t}^*(\bar{g})$ on $(a, T]$,
\[
	|\Rm|^2_{g(t)}(\phi^{-1}_{t}(x_0)) = \frac{1}{(\lambda(t- T) + 1)^2}|\Rmb|^2_{\bar{g}}(x_0).
\]  
which implies $\limsup_{t\to a}|\Rm|^2_{g(t)} = \infty$.
So we must have $a < 0$.
\end{proof}
\begin{acknowledgement*}
	The author wishes to thank Professors Bennett Chow and Lei Ni for useful discussions and their advice 
	regarding an early draft of the paper.
\end{acknowledgement*}

\bibliographystyle{amsalpha}

\end{document}